\documentclass[final]{dmtcs-episciences}
\usepackage[utf8]{inputenc}
\usepackage{subfigure}
\usepackage[round]{natbib}
\usepackage{amsmath,amsthm,graphicx,hyperref}
\usepackage{enumerate}
\newtheorem{theorem}{Theorem}

\newcommand{\la}{\lambda}
\newcommand{\La}{\Lambda}

\newcommand{\ris}{\mathrm{ris}}
\newcommand{\risT}{\mathrm{risT}}
\newcommand{\risB}{\mathrm{risB}}
\newcommand{\risL}{\mathrm{risL}}
\newcommand{\risA}{\mathrm{risA}}

\newcommand{\nth}[1][n]{{#1}^{\mathrm{th}}}

\newcommand{\sg}{\sigma}

\newcommand{\cref}[1]{Corollary \ref{corollary:#1}}

\newcommand{\red}{\mathrm{red}}

\author{Jeffrey Remmel\affiliationmark{1}\thanks{Email: remmel@math.ucsd.edu}
  \and Sainan Zheng\affiliationmark{2}\thanks{Corresponding author. Email: zhengsainandlut@hotmail.com}}
\title[Rises in forests of binary shrubs]{Rises in forests of binary shrubs
\footnote{The second author would like to thank the China Scholarship Council for financial
support. Her work was done during her visit to the Department of Mathematics,
University of California, San Diego.}}

\affiliation{
  Department of Mathematics, University of California, San Diego, USA\\
  School of Mathematical Sciences, Dalian University of Technology, PR China}
\received{2016-11-30}
\accepted{2017-5-4}
\revised{2017-5-4}
\begin{document}
\publicationdetails{19}{2017}{1}{15}{2561}
\maketitle
\begin{abstract}
\noindent
The study of patterns in permutations associated with
forests of binary shrubs was initiated
by Bevan, Levin, Nugent, Pantone, Pudwell, Riehl, and Tlachac.
In this paper, we study five different types of rise statistics
that can be associated with such permutations and
find the generating functions for the distribution of such rise statistics.
\end{abstract}

\section{Introduction}

In \cite{BLNPPRT},  the study of patterns in forests of binary shrubs was introduced.
A $k$-ary heap $H$ is a $k$-ary tree labeled
with $\{1, \ldots, n\}$ such that every child has a larger label
than its parent.  Given a $k$-ary heap $H$, we associate
a permutation $\sg_H$ with $H$ by recording the vertex labels as
they are encountered  in the breadth-first search of the tree.
For example, in Figure~\ref{fig:3Heap}, we picture a 3-ary heap $H$ whose
associated permutation is $\sg_H = 1~6~2~3~7~10~8~9~5~4$.

\begin{figure}[htbp]
  \begin{center}
    \includegraphics[width=0.4\textwidth,height=2cm]{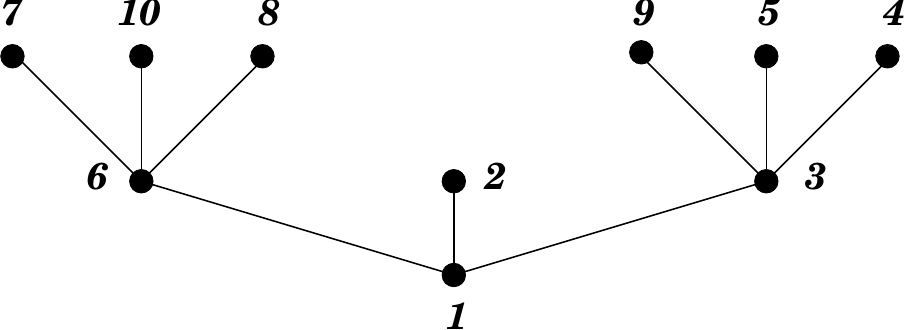}
    \caption{A 3-ary Heap.}
    \label{fig:3Heap}
  \end{center}
\end{figure}

A shrub is a heap whose leaves are all at most distance 1 from
the root.  A binary shrub is a heap whose underlying tree is
a shrub with three vertices. A binary shrub forest is
an ordered sequence of binary shrubs and we let
$\mathcal{F}_n^2$ denote the set of all forests
$F=(F_1, \ldots, F_n)$ of $n$ binary
shrubs whose set of labels is $\{1, \ldots, 3n\}$. For example,
in Figure \ref{fig:Fschrub}, we picture an element
of $\mathcal{F}_5^2$. Given a forest
$F =(F_1, \ldots,F_n) \in \mathcal{F}^2_n$,
we let $\sg_F$ denote the permutation that results by concatenating
the permutations $\sg_{F_1} \ldots \sg_{F_n}$. For example,
the permutation $\sg_F$ for the $F \in \mathcal{F}_5^2$ pictured
in Figure~\ref{fig:Fschrub} is
$$\sg_F = 5~12~9~6~13~15~1~4~10~7~11~8~2~14~3.$$
For any $n \geq 1$, we let $\mathcal{SF}_n^2$ denote the set of all
$\sg_F$ such that $F \in \mathcal{F}_n^2$.

\begin{figure}[htbp]
  \begin{center}
    \includegraphics[width=0.5\textwidth,height=1.5cm]{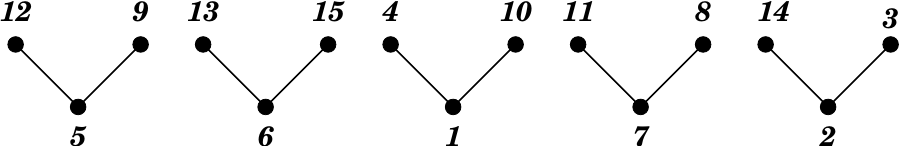}
    \caption{An element of $\mathcal{F}_n^2$.}
    \label{fig:Fschrub}
  \end{center}
\end{figure}

The goal of this paper is to study generating functions
for various types of rises  in $\mathcal{SF}_n^2$. For example,
given a permutation $\sg = \sg_1 \cdots \sg_n$ in the symmetric
group $S_n$, we let
$$Rise(\sg) = \{i:\sg_i < \sg_{i+1}\} \ \mbox{and} \ \ris(\sg) = |Rise(\sg)|.$$
For any sequence $\vec{a} = a_1 \cdots a_n$ of pairwise distinct positive integers,
we let the {\em reduction} of $\vec{a}$, $\red(\vec{a})$, be the permutation of $S_n$ that arises
from $\vec{a}$ by replacing $i^{\mathrm{th}}$-smallest element
of $\{a_1, \ldots, a_n\}$ by $i$. For example,
$\red(7~9~4~2~10) = 3~4~2~1~5$.

Now suppose that we are given $F =(F_1, \ldots,F_n) \in \mathcal{F}^2_n$,
then we let $\ris(F) = \ris(\sg_F)$. However, given
the structure of $F$, there are many other natural notions of rises
in  a forest of binary shrubs. That is, suppose
that $\sg_{F_i} = abc$ and $\sg_{F_{i+1}} = def$ as pictured in Figure
\ref{fig:2schrubs}.  Then we shall consider the following four types
of rises.
\begin{enumerate}
\item $F_i <_T F_{i+1}$ if every element of $\{a,b,c\}$ is less
than every element of $\{d,e,f\}$. We will refer to this type of
rise as {\em total rise}.
\item $F_i <_B F_{i+1}$ if $a < d$. We will refer to this type of
rise as {\em base rise}.
\item $F_i <_L F_{i+1}$ if $a < d$, $b < e$, and $c < f$.
We will refer to this type of rise as {\em lexicographic rise}.
\item $F_i <_A F_{i+1}$ if $c < e$. We refer to this type of rise
as an {\em adjacent rise} because when we look at the pictures of
$F_i$ and $F_{i+1}$, the rightmost element of $F_i$ is less then the leftmost element
of $F_{i+1}$.
\end{enumerate}
Then we define

\begin{xalignat*}{2}
RiseT(F) &= \{i:F_i <_T F_{i+1}\}&  \risT(F) &= |RiseT(F)|, \\
RiseB(F) &= \{i:F_i <_B F_{i+1}\}&  \risB(F) &= |RiseB(F)|, \\
RiseL(F) &= \{i:F_i <_L F_{i+1}\}& \risL(F) &= |RiseL(F)|, \ \mbox{and} \\
RiseA(F) &= \{i:F_i <_A F_{i+1}\}&  \risA(F) &= |RiseA(F)|.
\end{xalignat*}

\begin{figure}[htbp]
  \begin{center}
    \includegraphics[width=0.3\textwidth,height=1.5cm]{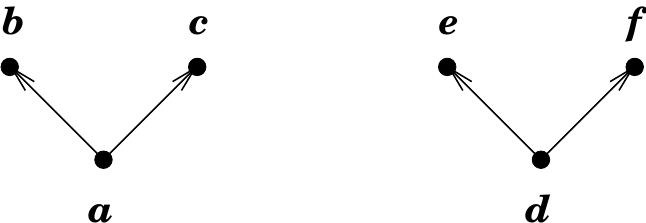}
    \caption{Two consecutive binary shrubs.}
    \label{fig:2schrubs}
  \end{center}
\end{figure}

The goal of this paper is to study the following generating functions.

\begin{eqnarray*}
\mathcal{R}(x,t)  &=& 1 + \sum_{n \geq 1} \frac{t^{3n}}{(3n!)}
\sum_{\sg \in \mathcal{SF}_n^2} x^{\ris(\sg)},\\
\mathcal{RT}(x,t)  &=& 1 + \sum_{n \geq 1} \frac{t^{3n}}{(3n!)}
\sum_{F \in \mathcal{F}_n^2} x^{\risT(F)}, \\
\mathcal{RB}(x,t)  &=& 1 + \sum_{n \geq 1} \frac{t^{3n}}{(3n!)}
\sum_{F \in \mathcal{F}_n^2} x^{\risB(F)}, \\
\mathcal{RL}(x,t)  &=& 1 + \sum_{n \geq 1} \frac{t^{3n}}{(3n!)}
\sum_{F \in \mathcal{F}_n^2} x^{\risL(F)}, \mbox{and} \\
\mathcal{RA}(x,t)  &=& 1 + \sum_{n \geq 1} \frac{t^{3n}}{(3n!)}
\sum_{F \in \mathcal{F}_n^2} x^{\risA(F)}.
\end{eqnarray*}
For example, we shall prove that
\begin{equation}\label{risgf}
\mathcal{R}(x,t) = \frac{1-x}{1-x + \sum_{n \geq 1}
\frac{(x(x-1)t^3)^n}{(3n)!} \prod_{k=1}^n (x+3k-2)}.
\end{equation}

To compute the remaining generating functions, we will need
find explicit formulas for the number of increasing binary
shrub forests relative to the orderings $<_T, <_B, <_L$ and
$<_A$. For $Z \in \{T,B,L,A\}$, we let
\begin{eqnarray*}
\mathcal{IZF}_n^2 &=& \{(F_1, \ldots, F_n) \in \mathcal{F}_n^2:
F_1 <_Z F_2 <_Z \cdots <_Z F_n\}, \\
\mathrm{IZF}_n^2 &=& |\mathcal{IZF}_n^2|, \ \mbox{and} \\
\mathcal{IZSF}_n^2 &=& \{\sg_F: F \in \mathcal{IZF}_n^2\}.
\end{eqnarray*}
Then  for $Z \in \{T,B,L,A\}$, we shall show that
\begin{eqnarray}\label{eq:Z}
\mathcal{RZ}(x,t) &=& 1 + \sum_{n \geq 1} \frac{t^{3n}}{(3n)!}
\sum_{F \in \mathcal{F}_n^2} x^{\mathrm{risZ}(F)} \nonumber \\
&=& \frac{1}{1- \sum_{n \geq 1} \frac{t^{3n}}{(3n)!}(x-1)^{n-1}
\mathrm{IZF}_n^2}.
\end{eqnarray}
Thus to find the generating functions
$\mathcal{RT}(x,t)$, $\mathcal{RB}(x,t)$, $\mathcal{RL}(x,t)$, and
 $\mathcal{RA}(x,t)$, we need only compute $\mathrm{ITF}_n^2$,
$\mathrm{IBF}_n^2$, $\mathrm{ILF}_n^2$,  and $\mathrm{IAF}_n^2$.
We shall show that
\begin{eqnarray*}
\mathrm{ITF}_n^2 &=& 2^n, \\
\mathrm{IBF}_n^2 &=& \frac{(3n)!}{3^n n!}, \ \mbox{and} \\
\mathrm{ILF}_n^2 &=& \frac{4^n (3n)!}{(n+1)! (2n+1)!}.
\end{eqnarray*}
Of these three formulas, the most interesting is the formula
for $\mathrm{ILF}_n^2$ which equals the number of paths of length $n$
in the plane that start and end at the origin and
which stay in the first quadrant
that consists only of steps of the form
$(1,1)$, $(0,-1)$ and $(-1,0)$.  This number was first computed
by Kreweras, see \cite{K65}.  We shall prove our formula by providing
a bijection between $\mathcal{ILF}^2_n$ and the collection of such paths.
We have not been able to find an explicit formula for $\mathrm{IAF}_n^2$, but we shall show that we can develop a system of recurrences that will allow
us to compute $\mathrm{IAF}_n^2$.

The main tool that we will use to compute these generating functions
is the homomorphism method as described in \cite{MR}.  The homomorphism method derives generating functions for
various permutation statistics by
applying a ring homomorphism defined on the
ring of symmetric functions \begin{math}\Lambda\end{math}
in infinitely many variables \begin{math}x_1,x_2, \ldots \end{math}
to simple symmetric function identities such as
\begin{equation}\label{conclusion2}
H(t) = 1/E(-t)
\end{equation}
where $H(t)$ and $E(t)$ are the generating functions for the homogeneous and elementary
symmetric functions, respectively:
\begin{equation}\label{genfns}
H(t) = \sum_{n\geq 0} h_n t^n = \prod_{i\geq 1} \frac{1}{1-x_it},~~~~ E(t) = \sum_{n\geq 0} e_n t^n = \prod_{i\geq 1} 1+x_it.
\end{equation}

The outline of the this paper is as follows. First in Section 2,
we shall briefly review the background on symmetric functions that we need.
In Section 3, we shall prove (\ref{risgf}).
In Section 4, we shall prove (\ref{eq:Z}). In Section 5, we
will compute $\mathrm{ITF}_n^2$,
$\mathrm{IBF}_n^2$, $\mathrm{ILF}_n^2$,  and $\mathrm{IAF}_n^2$ which, when
combined with the results of Section 4,  will allow us to compute
the generating functions
$\mathcal{RT}(x,t)$, $\mathcal{RB}(x,t)$, $\mathcal{RL}(x,t)$, and
$\mathcal{RA}(x,t)$.

\section{Symmetric functions}

In this section, we give the necessary background on symmetric functions that will be used in our proofs.

A partition of $n$ is a sequence of positive integers \begin{math}\la = (\la_1, \ldots ,\la_k)\end{math} such that \begin{math}0 < \la_1 \leq \cdots \leq \la_k\end{math} and $n=\la_1+ \cdots +\la_k$. We shall write $\lambda \vdash n$ to denote that $\lambda$ is partition of $n$ and we let $\ell(\lambda)$ denote the number of parts of $\lambda$. When a partition of $n$ involves repeated parts, we shall often use exponents in the partition notation to indicate these repeated parts. For example, we will write $(1^2,4^5)$ for the partition $(1,1,4,4,4,4,4)$.

Let \begin{math}\Lambda\end{math} denote the ring of symmetric functions in infinitely  many variables \begin{math}x_1,x_2, \ldots \end{math}. The \begin{math}\nth\end{math} elementary symmetric function \begin{math}e_n = e_n(x_1,x_2, \ldots )\end{math}  and \begin{math}\nth\end{math} homogeneous symmetric function \begin{math}h_n = h_n(x_1,x_2, \ldots )\end{math} are defined by the generating functions given in (\ref{genfns}). For any partition \begin{math}\la = (\la_1,\dots,\la_\ell)\end{math}, let \begin{math}e_\la = e_{\la_1} \cdots e_{\la_\ell}\end{math} and \begin{math}h_\la = h_{\la_1} \cdots h_{\la_\ell}\end{math}.  It is well known that \begin{math}e_0,e_1, \ldots \end{math} is an algebraically independent set of generators for \begin{math}\La\end{math}, and hence, a ring homomorphism \begin{math}\theta\end{math} on \begin{math}\Lambda\end{math} can be defined  by simply specifying \begin{math}\theta(e_n)\end{math} for all \begin{math}n\end{math}.

If $\lambda =(\lambda_1, \ldots, \lambda_k)$ is a partition of $n$, then a $\lambda$-brick tabloid of shape $(n)$ is a filling of a rectangle consisting of $n$ cells with bricks of sizes $\lambda_1, \ldots, \lambda_k$ in such a way that no two bricks overlap. For example, Figure \ref{fig:DIMfig1} shows the six $(1^2,2^2)$-brick tabloids of shape $(6)$.

\begin{figure}[htbp]
  \begin{center}
    \includegraphics[width=0.5\textwidth,height=1.5cm]{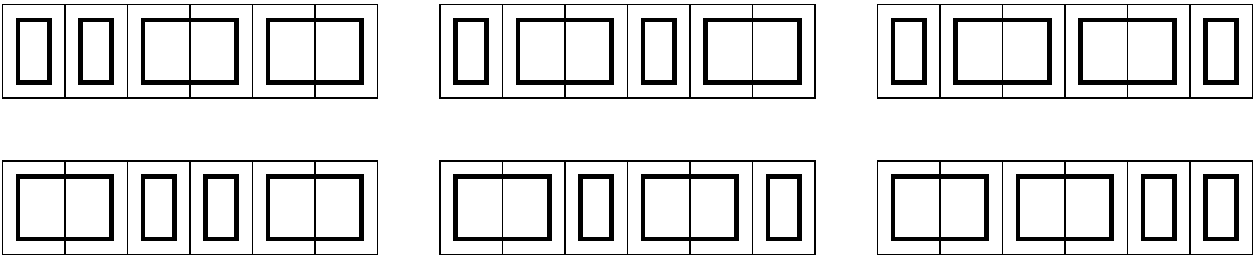}
    \caption{The six $(1^2,2^2)$-brick tabloids of shape $(6)$.}
    \label{fig:DIMfig1}
  \end{center}
\end{figure}

Let \begin{math}\mathcal{B}_{\la,n}\end{math} denote the set of \begin{math}\la\end{math}-brick tabloids of shape \begin{math}(n)\end{math} and let \begin{math}B_{\la,n}\end{math} be the number of \begin{math}\la\end{math}-brick tabloids of shape \begin{math}(n)\end{math}.  If \begin{math}B \in \mathcal{B}_{\la,n}\end{math}, we will write \begin{math}B =(b_1, \ldots, b_{\ell(\la)})\end{math} if the lengths of the bricks in \begin{math}B\end{math}, reading from left to right, are \begin{math}b_1, \ldots, b_{\ell(\la)}\end{math}. For example, the brick tabloid in the top right position in Figure \ref{fig:DIMfig1} is denoted as $(1,2,2,1)$. In \cite{Eg1} it has been proved that
\begin{equation}\label{htoe}
h_n = \sum_{\la \vdash n} (-1)^{n - \ell(\la)} B_{\la,n}~ e_\la.
\end{equation}

\section{The generating function \texorpdfstring{$\mathcal{R}(x,t)$}{}.}

It this section, we shall prove the following theorem.
\begin{theorem}\label{thm:Ris}
\begin{equation}\label{eq:Ris}
\mathcal{R}(x,t) = 1 + \sum_{n \geq 1} \frac{t^{3n}}{(3n)!}
\sum_{\sg \in \mathcal{SF}_n^2} x^{\ris(\sg)} =
\frac{1-x}{1-x + \sum_{n \geq 1}
\frac{(x(x-1)t^3)^n}{(3n)!} \prod_{k=1}^n (x+3k-2)}.
\end{equation}
\end{theorem}
\begin{proof}
Let $\mathbb{Q}[x]$ denote the polynomial ring over the rational
numbers $\mathbb{Q}$.

Let $\theta:\Lambda \rightarrow \mathbb{Q}[x]$ be the ring homomorphism defined on the ring of symmetric
functions $\Lambda$ in infinitely many variables determined by
setting $\theta(e_0) =1$, $\theta(e_{3n+1}) = \theta(e_{3n+2}) =0$
for all $n \geq 0$, and
$$\theta(e_{3n}) = \frac{(-1)^{3n-1}}{(3n)!} x^n (x-1)^{n-1}
\prod_{k=1}^n (x+3k-2)$$
for all $n \geq 1$.
We claim that for $n \geq 0$, $\theta(h_{3n+1}) = \theta(h_{3n+2}) =0$
and that for $n \geq 1$,
\begin{equation}\label{eq:ris1}
(3n)!\theta(h_{3n}) = \sum_{\sg \in \mathcal{SF}_n^2} x^{\ris(\sg)}.
\end{equation}

First it is easy to see that our definitions ensure that
$\theta(e_{\lambda}) =0$ if $\lambda$ has a part which
is equivalent to either 1 or 2 mod 3.
Since
\begin{equation}\label{eq:ris2}
h_n = \sum_{\lambda \vdash n} (-1)^{n - \ell(\lambda)}
B_{\lambda,n} e_{\lambda},
\end{equation}
it follows that $\theta(h_n) =0$ if $n$ is equivalent to 1 or 2 mod 3
since every partition of $\lambda$ of $n$ must contain a
part which is equivalent to 1 or 2 mod 3. If $\lambda = (\lambda_1,
\ldots, \lambda_k)$ is a partition of $n$, we let
$3 \lambda$ denote the partition $(3\lambda_1, \ldots, 3\lambda_k)$.
It follows that in the expansion $\theta(h_{3n})$, we need only
consider partitions $\lambda$ of $3n$ of the form
$3\mu$ where $\mu$ is a partition of $n$. Thus
\begin{eqnarray}\label{eq:ris3}
&&(3n)!\theta(h_{3n}) = (3n)! \sum_{\mu \vdash n}
(-1)^{3n -\ell(\mu)} B_{3\mu,3n} \theta(e_{3\mu}) = \nonumber \\
&&(3n)! \sum_{\mu \vdash n}
(-1)^{3n -\ell(\mu)}
\sum_{(3b_1, \ldots, 3b_{\ell(\mu)}) \in \mathcal{B}_{3\mu,3n}}
\prod_{i=1}^{\ell(\mu)} \frac{(-1)^{3b_i -1}}{(3b_i)!}
x^{b_i}(x-1)^{b_i-1} \prod_{k_i=1}^{b_i}(x+3k_i-2) = \nonumber \\
&&\sum_{\mu \vdash n}
\sum_{(3b_1, \ldots, 3b_{\mu}) \in \mathcal{B}_{3\mu,3n}}
\binom{3n}{3b_1, \ldots, 3b_{\ell(\mu)}} \prod_{i=1}^{\ell(\mu)}
x^{b_i}(x-1)^{b_i-1} \prod_{k_i=1}^{b_i}(x+3k_i-2).
\end{eqnarray}

Next our goal is to give a combinatorial interpretation for the right-hand
side of (\ref{eq:ris3}). Our combinatorial interpretation
will use a certain subset of permutations which are {\em increasing}
in a relevant way for our problem. In particular,
we let $\mathcal{ISF}_{n}^2$ equal the set of
all permutations
$\sg = \sg_1 \cdots \sg_{3n} \in \mathcal{SF}_{n}^2$ such
that $\sg_{3i}< \sg_{3i+1}$ for $i = 1, \ldots, n-1$. One way
to think of this set is that it is the set of permutations
that arise from a forest $F = (F_1, \ldots, F_n) \in \mathcal{F}_n^2$
such that the label of the right-most element in $F_i$ is less than the
label of the root
of $F_{i+1}$.  For example, if $n = 5$, then we are asking for
labellings of the poset whose Hasse diagram is
pictured at the top Figure \ref{fig:poset}.  We want to find
the set of all labellings of the nodes of this poset such
that when there is an arrow from a node $x$ to a node $y$, then
the label of node $x$ is less than label of node $y$. This is equivalent
to finding the set of all linear extensions of the poset.
We have given an example of such a labeling on the second line
of Figure \ref{fig:poset} and its corresponding permutation
in $\mathcal{SF}_5^2$ in the third line of Figure \ref{fig:poset}.
Given an element of $\sg = \sg_1 \cdots \sg_{3n} \in \mathcal{ISF}_{n}^2$,
we let
$$\ris_{1,2}(\sg) = |\{i:\sg_i < \sg_{i+1} \ \& \ i \equiv 1,2 \mod 3\}|.$$
That is, $\ris_{1,2}(\sg)$  keeps track of the number of
rises between pairs of the form $\sg_{3j+1} \sg_{3j+2}$ and
$\sg_{3j+2} \sg_{3j+3}$.

\begin{figure}[htbp]
  \begin{center}
    \includegraphics[width=0.5\textwidth,height=4cm]{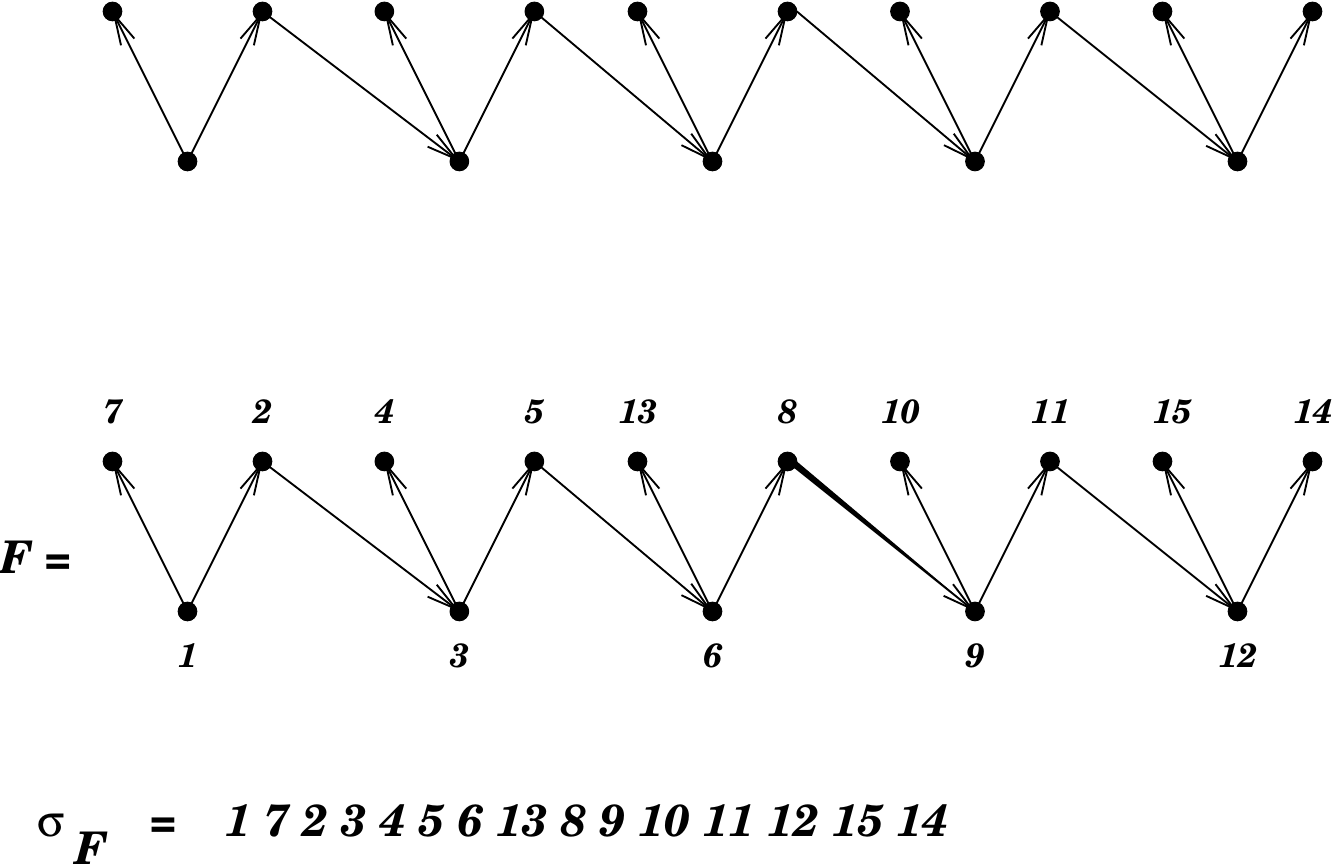}
    \caption{The poset for $\mathcal{ISF}_5^2$.}
    \label{fig:poset}
  \end{center}
\end{figure}

We claim that
$$x^n\prod_{k=1}^n (x+3k-2) = \sum_{\sg \in \mathcal{ISF}_n^2}
x^{\ris_{1,2}(\sg)}.$$
This is easy to prove by induction.
First, it easy to check that there are exactly two permutations
in $\mathcal{ISF}_1^2$, namely 123 and 132, so
that $\sum_{\sg \in \mathcal{ISF}_1^2} x^{\ris_{1,2}(\sg)} =x(1+x)$
as claimed. Now suppose that our formula holds for $k < n$.
Then consider Figure \ref{fig:rising1} where we have redrawn
the poset so that the positions correspond to the
elements in $\sg_F$. It is easy to see that the label
of the left-most element must be one since there is a directed path
from that element to any other element in the poset. There must be a rise
from $\sg_1$ to $\sg_2$ so we add a label $x$ below that position.
Next consider node which
has label 2.  If 2 is the label of the second element, then
the label of the third element must be 3 since there is a directed path
from that element to any of the other unlabeled elements in the poset at this point.
In this case $2=\sg_2 < \sg_3 =3$ so we add a label $x$ below that position.
If the label of the second element is $a$ where $a >2$, then
the label of the third element must be 2 since there is a directed path
from that element to any of the other unlabeled elements in the poset at this point.
We have $3n-2$ ways to choose $a$. In this
case the pair $\sg_2 \sg_3$ is not a rise so that that we do
not add a label $x$ below that position. Thus our
choices of labels for the binary shrub $F_1$ gives rise to
a factor of $x(x+3n-2)$ in our sum. Note that once we have
placed the labels on $F_1$, the remaining labels are completely free.
Thus it follows that
\begin{eqnarray*}
\sum_{\sg \in \mathcal{ISF}_n^2}
x^{\ris_{1,2}(\sg)} &=& x(x+3n-2) \sum_{\sg \in \mathcal{ISF}_{n-1}^2} x^{\ris_{1,2}(\sg)} \\
&=& x^n \prod_{k=1}^{n} (x+3k-2).
\end{eqnarray*}

\begin{figure}[htbp]
  \begin{center}
    \includegraphics[width=0.5\textwidth,height=4cm]{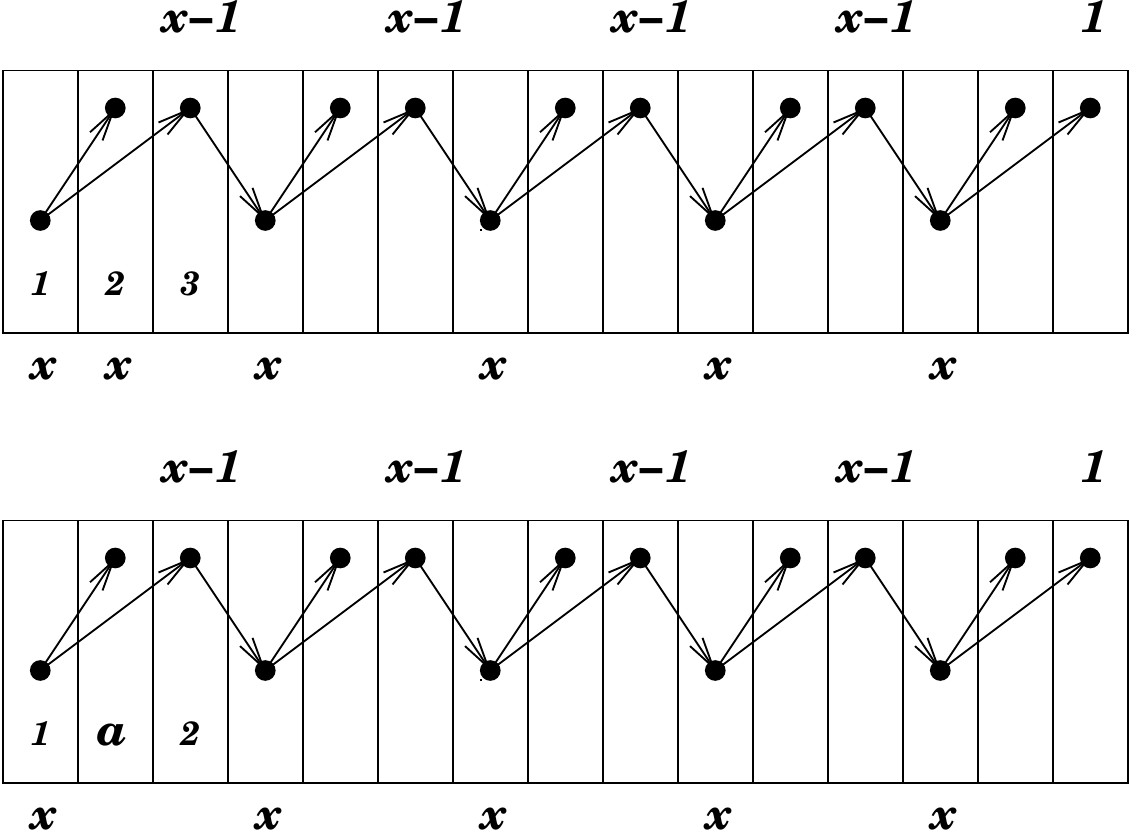}
    \caption{The recursive construction of elements of
     $\mathcal{ISF}_n^2$.}
    \label{fig:rising1}
  \end{center}
\end{figure}

To complete our combinatorial interpretation for the right-hand
side of (\ref{eq:ris3}), we  interpret
the extra factor of $(x-1)^{n-1}$ in $\theta(e_{3n})$
as adding a label $(x-1)$ on every third element
except the last one. In Figure \ref{fig:rising1}, we indicate
this by putting such labels at the top of the diagram.

We are now in a position to give a combinatorial interpretation
to the right-hand side of (\ref{eq:ris3}).  That is,
we first choose a brick tabloid $B = (3b_1, \ldots, 3b_{\ell(\mu)})$
consisting of bricks whose size is a multiple of 3.
Then we use the multinomial
coefficient $\binom{3n}{3b_1, \ldots, 3b_{\ell(\mu)}}$ to pick
an ordered sequence of sets $S_1, \ldots, S_{\ell(\mu)}$ such
that $|S_i| =3b_i$ and $S_1, \ldots, S_{\ell(\mu)}$ partition
the elements $\{1, \ldots, 3n\}$. For each brick $3b_i$, we
interpret  the factor $x^{b_i} \prod_{k=1}^{b_i} (x+3k-2)$
as all ways $\gamma_1^{(i)} \ldots \gamma^{(i)}_{3b_i}$
of arranging the elements of $S_i$ in the cells
of the brick $3b_i$   such
that $\red(\gamma_1^{(i)} \cdots \gamma^{(i)}_{3b_i}) \in
\mathcal{ISF}_{b_i}^2$ where we place a label $x$
below the cell containing $\gamma^{(i)}_j$ if $j =1,2 \mod 3$ and
$\gamma^{(i)}_j < \gamma^{(i)}_{j+1}$. Finally, we interpret
the factor $(x-1)^{b_i-1}$ as all ways of labeling
the cells containing the elements
$\gamma^{(i)}_3, \ldots, \gamma^{(i)}_{3b_i -3}$ with
either $x$ or $-1$. We shall also label the last cell of a brick
$3b_i$ with 1.
Let $\mathcal{O}_{3n}$ denote the set of all objects created in this
way. Then $\mathcal{O}_{3n}$ consists of all triples
$(B,\sg,L)$ such that $B =(3b_1, \ldots, 3b_k)$ is a brick
tabloid all of whose bricks have length a multiple of 3,
$\sg$ is a permutation in $S_{3n}$, and $L$ is labeling of
the cells of $B$  such that the following four conditions hold.
\begin{enumerate}
\item For each $i =1, \ldots ,k$, the reduction of the sequence
of elements obtained by reading the elements in the brick $3b_i$
from left to right is an element of $\mathcal{ISF}_{b_i}^2$.
\item The cell containing a $\sg_i$ such that $i \equiv 1,2 \mod 3$ is
labeled with an $x$ if and only if $i \in Rise(\sg)$.
\item The label of a cell at the end of any brick is $1$.
\item The cells containing elements of the form
 $\sg_{3i}$ which are not at the end
of a brick are labeled with either $-1$ or $x$.
\end{enumerate}

For each such $(B,\sg,L) \in \mathcal{O}_{3n}$, we let the weight
of $(B,\sg,L)$,
$w(B,\sg,L)$, be the product of all its $x$ labels and we let
the sign of $(B,\sg,L)$, $sgn(B,\sg,L)$, be the product of all its
$-1$ labels. For example, at the top of Figure \ref{fig:rising2}, we
picture an element $(B,\sg,L) \in \mathcal{O}_{18}$ such that
$w(B,\sg,L) = x^{11}$ and $sgn(B,\sg,L) =-1$. It follows
that
\begin{equation}\label{eq:ris4}
(3n)!\theta(h_{3n}) = \sum_{(B,\sg,L) \in \mathcal{O}_{3n}}
sgn(B,\sg,L)w(B,\sg,L).
\end{equation}

Next we will define a sign-reversing involution
$J:\mathcal{O}_{3n} \rightarrow \mathcal{O}_{3n}$ which we will use
to simplify the right-hand side of (\ref{eq:ris4}).
Given a triple $(B,\sg,L) \in \mathcal{O}_{3n}$, where $B =(3b_1,
\ldots, 3b_k)$ and $\sg = \sg_1 \cdots \sg_{3n}$, scan the cells
from left to right looking for the first cell $c$ such that either
\begin{enumerate}[{Case }1:]
\item  $c = 3s$ for some $1 \leq s \leq n-1$
and the label on cell $c$ is $-1$ or
\item $c$ is that last cell of brick $3b_i$ for some $i <k$
and $\sg_c < \sg_{c+1}$.
\end{enumerate}

In Case 1, suppose that $c$ is in brick $3b_i$. Then $J(B,\sg,L)$
is obtained from $(B,\sg,L)$ by splitting
brick $3b_i$ into two bricks $3b_i^*$ and $3b_i^{**}$, where
$3b_i^*$ contains the cells of $3b_i$ up to and including
cell $c$ and $3b_i^{**}$ contains the remaining cells of $3b_i$,
and changing the label on cell $c$ from $-1$ to 1.  In Case 2,
$J(B,\sg,L)$
is obtained from $(B,\sg,L)$ by combining  bricks $3b_{i}$ and $3b_{i+1}$ into a single brick
$3b$ and changing the label on cell $c$ from 1 to $-1$. If neither
Case 1 or Case 2 applies, then we define $J(B,\sg,L) = (B,\sg,L)$.

\begin{figure}[htbp]
  \begin{center}
    \includegraphics[width=0.5\textwidth,height=4.5cm]{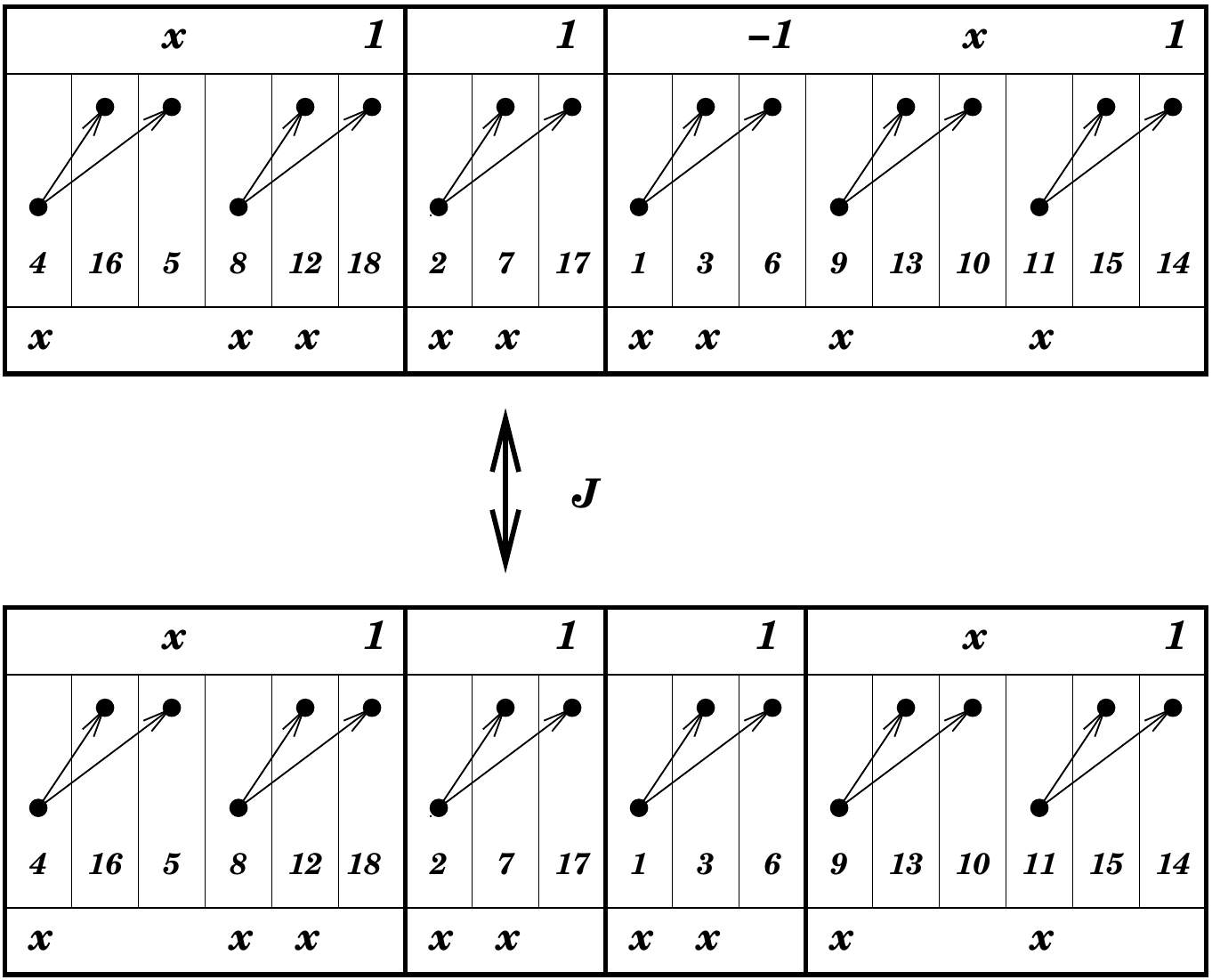}
    \caption{An example of the  involution $J$.}
    \label{fig:rising2}
  \end{center}
\end{figure}

For example,
if $(B,\sg,L)$ is the element of $\mathcal{O}_{18}$ pictured
at the top of Figure \ref{fig:rising2}, then
$B=(3b_1,3b_2,3b_3)$ where $b_1 =2$, $b_2=1$ and $b_3 =3$. Note that
we cannot combine bricks $3b_1$ and $3b_2$ since $18 =\sg_6 > \sg_7 =2$
and we cannot combine bricks $3b_2$ and $3b_3$ since $17 =\sg_9 > \sg_{10} =1$.
Thus the first cell $c$ where either Case 1 or Case 2 applies is
cell  $c=12$. Thus we are in Case 1 and $J(B,\sg,L)$ is
obtained from $(B,\sg,L)$ by splitting
brick $3b_3$ into two bricks, the first one of size 3 and second one of
size 6, and changing the label on cell 12 from $-1$ to 1. Thus
$J(B,\sg,L)$ is pictured at the bottom of Figure \ref{fig:rising2}.

It is easy to see that $J$ is an involution. That is, if we are
in Case I using cell $c$ to define $J(B,\sg,L)$, then we will be
in Case II using cell $c$ when we apply $J$ to  $J(B,\sg,L)$ so
that $J(J(B,\sg,L)) = (B,\sg,L)$. Similarly, if we are
in Case II using cell $c$ to define $J(B,\sg,L)$, then we will be
in Case I using cell $c$ when we apply $J$ to  $J(B,\sg,L)$ so
that $J(J(B,\sg,L)) = (B,\sg,L)$. Moreover it is easy to see
that if $J(B,\sg,L) \neq (B,\sg,L)$, then
$$sgn(B,\sg,L)w(B,\sg,L) = -sgn(J(B,\sg,L))w(J(B,\sg,L)).$$
It follows that
\begin{eqnarray}\label{eq:ris5}
(3n)!\theta(h_{3n}) &=& \sum_{(B,\sg,L) \in \mathcal{O}_{3n}}
sgn(B,\sg,L)w(B,\sg,L) \nonumber \\
&=& \sum_{(B,\sg,L) \in \mathcal{O}_{3n},
J(B,\sg,L) = (B,\sg,L)}
sgn(B,\sg,L)w(B,\sg,L).
\end{eqnarray}

Thus we must examine the fixed points of $J$ on $\mathcal{O}_{3n}$.
It is easy to see that if $(B,\sg,L)$, where $B=(3b_1, \ldots, 3b_k)$ and
$\sg = \sg_1 \cdots \sg_{3n}$, is a fixed point of $J$, then
there can be no cells labeled $-1$ and for $1 \leq i \leq k-1$,
the element in the last cell of brick $3b_i$ must be greater than the
element in the first cell of $3b_{i+1}$.  It follows that if
$c =3i$ for some $1 \leq i \leq n-1$, then
cell $c$ is labeled with an $x$ if and only if $\sg_c < \sg_{c+1}$.
Thus for a fixed point $(B,\sg,L)$ of $J$, $wt(B,\sg,L) = x^{\ris(\sg)}$
and $sgn(B,\sg,L) =1$. On the other hand, given any
$\sg \in \mathcal{SF}_n^2$, we can create a fixed point $(B,\sg,L)$ of
$J$ by having the bricks of $B$ end at the cells $c=3i$
such that $3i \not \in Rise(\sg)$ and labeling all the cells
$j$ such that $j \in Rise(\sg)$ with an $x$. For example,
if
$$\sg = 4~16~5~8~12~18~2~7~17~1~3~6~9~13~10~11~15~14,$$
then the fixed point corresponding to $\sg$ is pictured in
Figure \ref{fig:rising3}.

\begin{figure}[htbp]
  \begin{center}
    \includegraphics[width=0.5\textwidth,height=1.7cm]{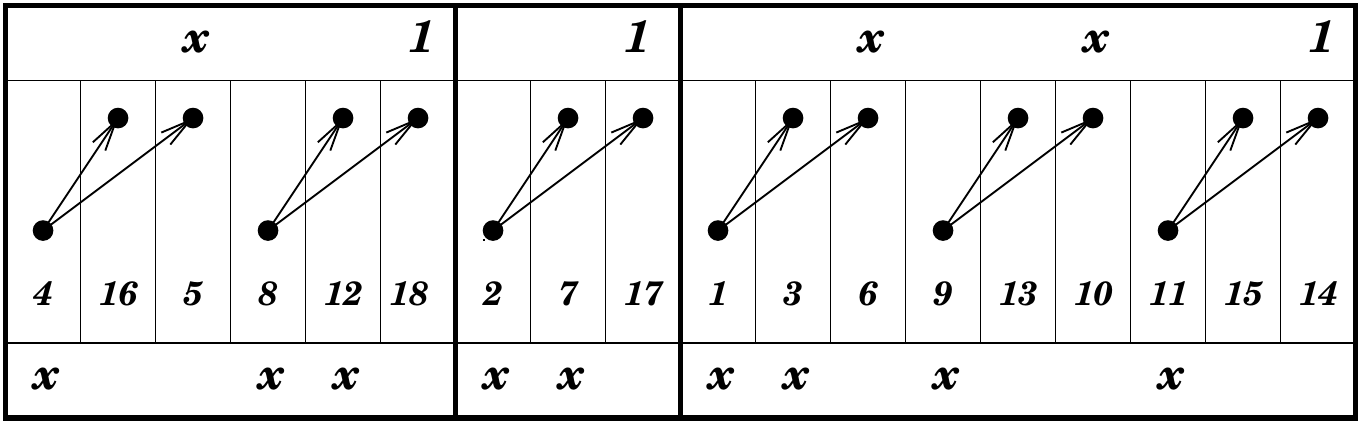}
    \caption{A fixed point of $J$.}
    \label{fig:rising3}
  \end{center}
\end{figure}

Hence, we have proved that
$$(3n)!\theta(h_{3n}) = \sum_{\sg \in \mathcal{SF}_n^2} x^{\ris(\sg)}$$
as desired.

It follows that
\begin{eqnarray*}
\theta(H(t)) &=& 1 + \sum_{n \geq 1}
\frac{t^{3n}}{(3n)!} \sum_{\sg \in \mathcal{SF}_n^2} x^{\ris(\sg)} \\
&=& \frac{1}{\theta(E(-t))} = \frac{1}{1+ \sum_{n\geq 1} (-t)^n \theta(e_{n})} \\
&=& \frac{1}{1+ \sum_{n\geq 1} (-t)^{3n} \frac{(-1)^{3n-1}}{(3n)!}
 x^n(x-1)^{n-1} \prod_{k=1}^n (x+3k-2)}\\
&=& \frac{1-x}{1-x + \sum_{n \geq 1} \frac{(x(x-1)t^3)^n}{(3n)!}\prod_{k=1}^n (x+3k-2)}.
\end{eqnarray*}

\end{proof}

We have used this generating function to compute the initial terms
of $R(x,t)$.
\begin{eqnarray*}
&&1+ (x (1+x)) \frac{t^3}{3!}+ x^2 \left(16+39 x+24 x^2+x^3\right)
\frac{t^6}{6!}+\\
&&x^3 \left(1036+4183 x+5506 x^2+2536 x^3+178 x^4+x^5\right)\frac{t^9}{9!} \\
&& x^4 \left(174664+992094 x+2054131 x^2+1896937 x^3+726622 x^4+67768 x^5+1383 x^6+x^7\right)\frac{t^{12}}{12!} \\&&x^5 \left(60849880+446105914 x+1272918569 x^2+1800188609 x^3+1307663949 x^4+\right.\\
&& \ \ \ \left. 442673265 x^5+49244651 x^6+1720211 x^7+10951 x^8+x^9\right)
\frac{t^{15}}{15!} + \cdots .
\end{eqnarray*}

We note that if $\sg = \sg_1 \cdots \sg_{3n} \in
\mathcal{SF}_n^2$, then we are forced to have
$\{3k+1:k=0, \ldots, n-1\} \subseteq Rise(\sg)$ by our definition
of the permutation associated with a forest of binary shrubs.
It follows that
\begin{eqnarray*}
\mathcal{R}(x,\frac{t}{x^{1/3}}) &=& 1 + \sum_{n \geq 1}
\frac{t^{3n}}{(3n)!} \sum_{\sg \in \mathcal{SF}_n^2} x^{\ris(\sg)-n}\\
&=& \frac{1-x}{1-x + \sum_{n \geq 1} \frac{((x-1)t^3)^n}{(3n)!}\prod_{k=1}^n (x+3k-2)}.
\end{eqnarray*}
We can then set $x=0$ in this expression to get the generating
function of $\sg \in \mathcal{SF}_n^2$ such that $\ris(\sg) = n$ which
is the minimal number of rises that an element
$F \in \mathcal{SF}_n^2$ can have. That is,
\begin{eqnarray*}
1 + \sum_{n \geq 1}
\frac{t^{3n}}{(3n)!} |\{\sg \in \mathcal{SF}_n^2: \ris(\sg)=n\}|
&=&  \frac{1}{1 + \sum_{n \geq 1} \frac{(-t^3)^n}{(3n)!}\prod_{k=1}^n (3k-2)}\\
&=&\frac{1}{1 + \sum_{n \geq 1} \frac{(-1)^nt^{3n}}{(3n)!}\prod_{k=1}^n (3k-2)}.
\end{eqnarray*}

\section{The generating functions \texorpdfstring{$\mathcal{RZ}(x,t)$}{}
for \texorpdfstring{$Z \in \{T,B,L,A\}$}{}}

In this section, we shall give a general method for computing the generating
functions
$\mathcal{RT}(x,t)$,
$\mathcal{RB}(x,t)$, $\mathcal{RL}(x,t)$, and $\mathcal{RA}(x,t)$.
Recall that for $Z \in \{T,B,L,A\}$
\begin{eqnarray*}
\mathcal{IZF}_n^2 &=& \{(F_1, \ldots, F_n) \in \mathcal{F}_n^2:
F_1 <_Z F_2 <_Z \cdots <_Z F_n\}, \\
\mathrm{IZF}_n^2 &=& |\mathcal{IZF}_n^2|, \ \mbox{and} \\
\mathcal{IZSF}_n^2 &=& \{\sg_F: F \in \mathcal{IZF}_n^2\}.
\end{eqnarray*}
Then we have the following theorem.
\begin{theorem}\label{thm:gen}
For $Z \in \{T,B,A,L\}$,
\begin{equation}\label{eq:gen}
\mathcal{RZ}(x,t) =  1 + \sum_{n \geq 1} \frac{t^{3n}}{(3n)!}
\sum_{F \in \mathcal{F}_n^2} x^{\mathrm{risZ}(F)} =  \frac{1}{1- \sum_{n \geq 1} \frac{t^{3n}}{(3n)!}(x-1)^{n-1}
\mathrm{IZF}_n^2}.
\end{equation}
\end{theorem}
\begin{proof}
The proof of this theorem is similar to the proof of Theorem \ref{thm:Ris}.
The main difference between the two proofs is that in
Theorem \ref{thm:Ris}, we needed to keep track of the rises that
occur within each binary shrub in a forest
while in the current situation, we need only keep track
of the ``rises'' between adjacent binary shrubs  in a forest.

Let $Z \in \{T,B,L,A\}$ and let $\theta_Z:\Lambda \rightarrow \mathbb{Q}[x]$
be
the ring homomorphism determined by
setting $\theta_Z(e_0) =1$, $\theta_Z(e_{3n+1}) = \theta_Z(e_{3n+2}) =0$
for all $n \geq 0$, and
$$\theta_Z(e_{3n}) = \frac{(-1)^{3n-1}}{(3n)!}
\mathrm{IZF}_n^2 (x-1)^{n-1}$$
for all $n \geq 1$.
We claim that for $n \geq 0$, $\theta_Z(h_{3n+1}) = \theta_Z(h_{3n+2}) =0$
and that for $n \geq 1$,
\begin{equation}\label{eq:risZ1}
(3n)!\theta_Z(h_{3n}) = \sum_{F \in \mathcal{F}_n^2} x^{\mathrm{risZ}(F)}.
\end{equation}

We can use that same argument as in Theorem \ref{thm:Ris}
to conclude that $\theta_Z(h_n) =0$ if $n$ is equivalent to 1 or 2 mod 3
and that in the expansion $\theta(h_{3n})$, we need only
consider partitions $\lambda$ of $3n$ of the form
$3\mu$ where $\mu$ is a partition of $n$. Thus
\begin{eqnarray}\label{eq:risZ3}
(3n)!\theta_Z(h_{3n}) &=& (3n)! \sum_{\mu \vdash n}
(-1)^{3n -\ell(\mu)} B_{3\mu,3n} \theta_Z(e_{3\mu}) \nonumber \\
&=& (3n)! \sum_{\mu \vdash n}
(-1)^{3n -\ell(\mu)}
\sum_{(3b_1, \ldots, 3b_{\ell(\mu)}) \in \mathcal{B}_{3\mu,3n}}
\prod_{i=1}^{\ell(\mu)} \frac{(-1)^{3b_i -1}}{(3b_i)!} \mathrm{IZF}_{b_i}^2
(x-1)^{b_i-1}  \nonumber \\
&=& \sum_{\mu \vdash n}
\sum_{(3b_1, \ldots, 3b_{\mu}) \in \mathcal{B}_{3\mu,3n}}
\binom{3n}{3b_1, \ldots, 3b_{\ell(\mu)}} \prod_{i=1}^{\ell(\mu)}
 \mathrm{IZF}_{b_i}^2(x-1)^{b_i-1}.
\end{eqnarray}

As in the proof of Theorem \ref{thm:Ris}, we must give a combinatorial interpretation to the right-hand
side of (\ref{eq:risZ3}). We first choose a brick tabloid $B = (3b_1, \ldots, 3b_{\ell(\mu)})$
whose bricks have size a multiple of 3.  Then we use the multinomial
coefficient $\binom{3n}{3b_1, \ldots, 3b_{\ell(\mu)}}$ to pick
an ordered sequence of sets $S_1, \ldots, S_{\ell(\mu)}$ such
that $|S_i| =3b_i$ and $S_1, \ldots, S_{\ell(\mu)}$ partition
the elements $\{1, \ldots, 3n\}$. For each brick $3b_i$, we
interpret  the factor $\mathrm{IZF}_n^2$
as all ways $\gamma_1^{(i)} \cdots \gamma^{(i)}_{3b_i}$
of arranging the elements of $S_i$ in the cells
of the brick $3b_i$ such
that $\red(\gamma_1^{(i)} \cdots \gamma^{(i)}_{3b_i}) \in
\mathcal{IZSF}_{b_i}^2$.  Finally, we interpret
the factor $(x-1)^{b_i-1}$ as all ways of labeling
the cells containing the elements
$\gamma^{(i)}_3, \ldots, \gamma^{(i)}_{3b_i -3}$ with
either $x$ or $-1$. We also label the last cell of a brick with 1.
Let $\mathcal{OZ}_{3n}$ denote the set of all objects created in this
way. Then $\mathcal{OZ}_{3n}$ consists of all triples
$(B,\sg,L)$ such that $B =(3b_1, \ldots, 3b_k)$ is a brick
tabloid all of whose bricks have length a multiple of 3,
$\sg=\sg_1 \cdots \sg_{3n}$ is a permutation in $S_{3n}$, and $L$ is
labeling of
the cells of $B$  such that the following three conditions hold.
\begin{enumerate}
\item For each $i =1, \ldots ,k$, the reduction of the sequence
of elements obtained by reading the elements in the brick $3b_i$
from left to right is an element is in $\mathcal{IZSF}_{b_i}^2$.
 \item The label of a cell at the end of any brick is $1$.
\item The cells containing elements of the form
 $\sg_{3i}$ which are not at the end
of a brick are labeled with either $-1$ or $x$.
\end{enumerate}

For each such $(B,\sg,L) \in \mathcal{OZ}_{3n}$, we let the weight
of $(B,\sg,L)$,
$w(B,\sg,L)$, be the product of all its $x$ labels and we let
the sign of $(B,\sg,L)$, $sgn(B,\sg,L)$, be the product of all its
$-1$ labels. For example, suppose that $Z =B$. Then
at the top of Figure \ref{fig:rising2Z}, we
picture an element $(B,\sg,L) \in \mathcal{OB}_{18}$ such that
$w(B,\sg,L) = x^{2}$ and $sgn(B,\sg,L) =-1$.

It follows
that
\begin{equation}\label{eq:risZ4}
(3n)!\theta_Z(h_{3n}) = \sum_{(B,\sg,L) \in \mathcal{OZ}_{3n}}
sgn(B,\sg,L)w(B,\sg,L).
\end{equation}

Next we will define a sign-reversing involution
$J_Z:\mathcal{OZ}_{3n} \rightarrow \mathcal{OZ}_{3n}$ which we will use
to simplify the right-hand side of (\ref{eq:risZ4}).
Given a triple $(B,\sg,L) \in \mathcal{OZ}_{3n}$, where $B =(3b_1,
\ldots, 3b_k)$ and $\sg = \sg_1 \cdots \sg_{3n}$, scan the cells
from left to right looking for the first cell $c$ such that either
\begin{enumerate}[{Case }1:]
\item $c = 3s$ for some $1 \leq s \leq n-1$
and the label on cell $c$ is $-1$ or
\item $c$ is that last cell of brick $3b_i$ for some $i <k$
and the binary shrub $F$ corresponding to the cells
$3b_i-2,3b_i-1,3b_i$ is $<_Z$ the  binary shrub $G$ corresponding
to the cells
$3b_i+1,3b_i+2,3b_i+3$.
\end{enumerate}

In Case 1, suppose that $c$ is in brick $3b_i$. Then $J_Z(B,\sg,L)$
is obtained from $(B,\sg,L)$ by splitting
brick $3b_i$ into two bricks $3b_i^*$ and $3b_i^{**}$, where
$3b_i^*$ contains the cells of $3b_i$ up to and including
cell $c$ and $3b_i^{**}$ contains the remaining cells of $3b_i$,
and changing the label on cell $c$ from $-1$ to 1.  In Case 2,
$J_Z(B,\sg,L)$
is obtained from $(B,\sg,L)$ by combining  bricks $3b_{i}$ and $3b_{i+1}$
into a single brick
$3b$ and changing the label on cell $c$ from 1 to $-1$. If neither
Case 1 or Case 2 applies, then we define $J_Z(B,\sg,L) = (B,\sg,L)$.

\begin{figure}[htbp]
  \begin{center}
    \includegraphics[width=0.5\textwidth,height=4.5cm]{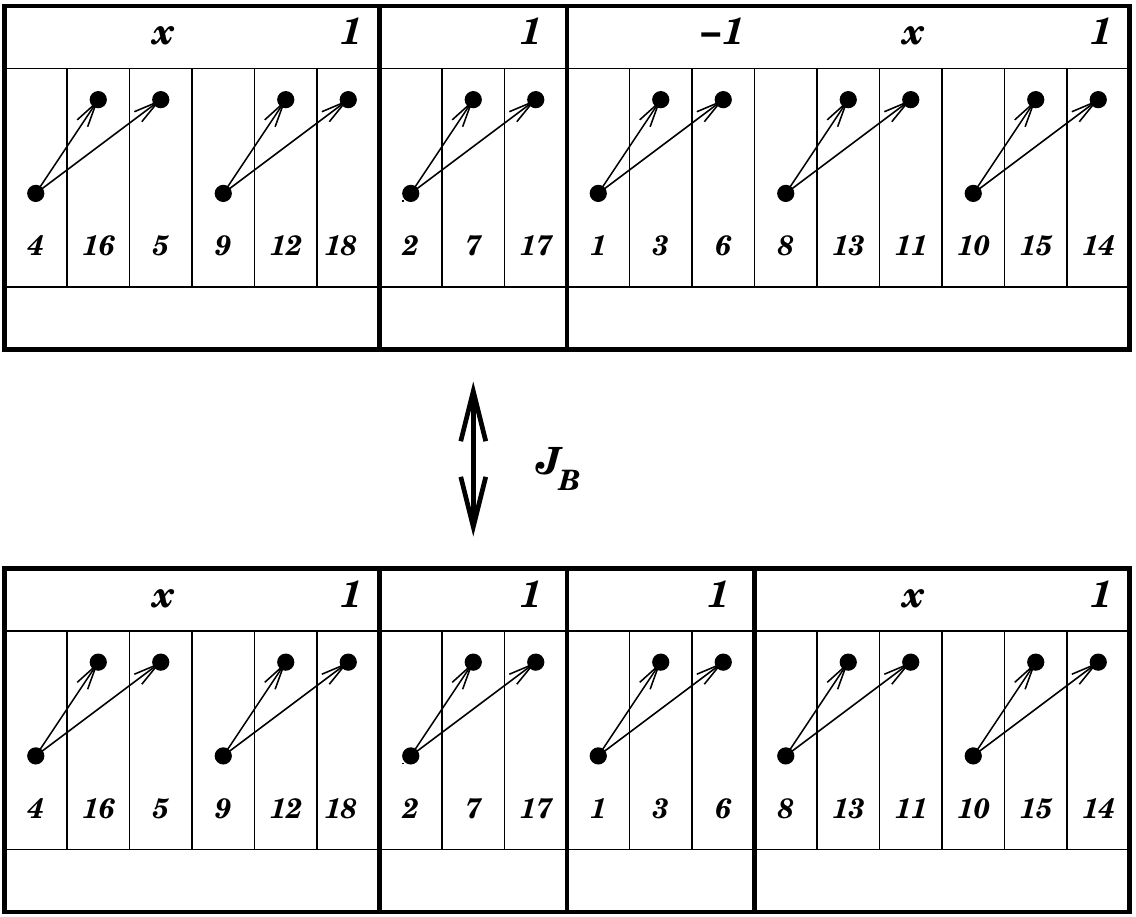}
    \caption{An example of the  involution $J_Z$ when $Z =B$.}
    \label{fig:rising2Z}
  \end{center}
\end{figure}

For example,
if $(B,\sg,L)$ is the element of $\mathcal{OB}_{18}$ pictured
at the top of Figure \ref{fig:rising2Z}, then
$B=(3b_1,3b_2,3b_3)$ where $b_1 =2$, $b_2=1$ and $b_3 =3$. Note that
we cannot combine bricks $3b_1$ and $3b_2$ since $9 =\sg_4 > \sg_7 =2$
and we cannot combine bricks $3b_2$ and $3b_3$ since $2 =\sg_7 > \sg_{10} =1$.
Thus the first cell $c$ where either Case 1 or Case 2 applies is
cell  $c=12$. Thus we are in Case 1 and $J_{B}(B,\sg,L)$ is obtained
from $(B,\sg,L)$ by splitting
brick $3b_3$ into two bricks, the first one of size 3 and second one of
size 6, and changing the label on cell 12 from $-1$ to 1. Thus
$J_B(B,\sg,L)$ is pictured at the bottom of Figure \ref{fig:rising2Z}.

We can use
the same reasoning as in Theorem \ref{thm:Ris} to show
that $J_Z$ is an involution.
Moreover it is easy to see
that if $J_Z(B,\sg,L) \neq (B,\sg,L)$, then
$$sgn(B,\sg,L)w(B,\sg,L) = -sgn(J_Z(B,\sg,L))w(J_Z(B,\sg,L)).$$
It follows that
\begin{eqnarray}\label{eq:ris5Z}
(3n)!\theta_Z(h_{3n}) &=& \sum_{(B,\sg,L) \in \mathcal{OZ}_{3n}}
sgn(B,\sg,L)w(B,\sg,L) \nonumber \\
&=& \sum_{(B,\sg,L) \in \mathcal{OZ}_{3n},
J_Z(B,\sg,L) = (B,\sg,L)}
sgn(B,\sg,L)w(B,\sg,L).
\end{eqnarray}

Thus we must examine the fixed points of $J_Z$ on $\mathcal{OZ}_{3n}$.
It is easy to see that if $(B,\sg,L)$, where $B=(3b_1, \ldots, 3b_k)$ and
$\sg = \sg_1 \cdots \sg_{3n}$, is a fixed point of $J_Z$, then
there can be no cells labeled $-1$ and for $1 \leq i \leq k-1$,
the binary shrub $F$ determined by the last three cells of
$3b_i$ is not $<_Z$ the binary shrub determined by the first
three cells of $3b_{i+1}$.  It follows that if
$c =3i$ for some $1\leq i \leq n-1$, then
cell $c$ is labeled with an $x$ if and only if
the binary shrub $F$ corresponding to the cells
$3b_i-2,3b_i-1,3b_i$ is $<_Z$ the  binary shrub $G$ corresponding
to the cells
$3b_i+1,3b_i+2,3b_i+3$.
Thus for a fixed point $(B,\sg,L)$ of $J_Z$, $wt(B,\sg,L) = x^{\mathrm{risZ}(\sg)}$
and $sgn(B,\sg,L) =1$. On the other hand, given any
$\sg_F$ where $F =(F_1, \ldots, F_n)
\in \mathcal{F}_n^2$, we can create a fixed point $(B,\sg_F,L)$ of
$J_Z$ by having the bricks of $B$ end at the cells $c=3i$
such that $i \not \in RiseZ(F)$ and labeling all the cells
$3j$ such that $j \in RiseZ(F)$ with an $x$. For example,
if $Z =B$ and
$$\sg = 4~16~5~8~12~18~2~7~17~1~3~6~9~13~10~11~15~14,$$
then the fixed point corresponding to $\sg$ is pictured in
Figure \ref{fig:rising3Z}.

\begin{figure}[htbp]
  \begin{center}
    \includegraphics[width=0.5\textwidth,height=2cm]{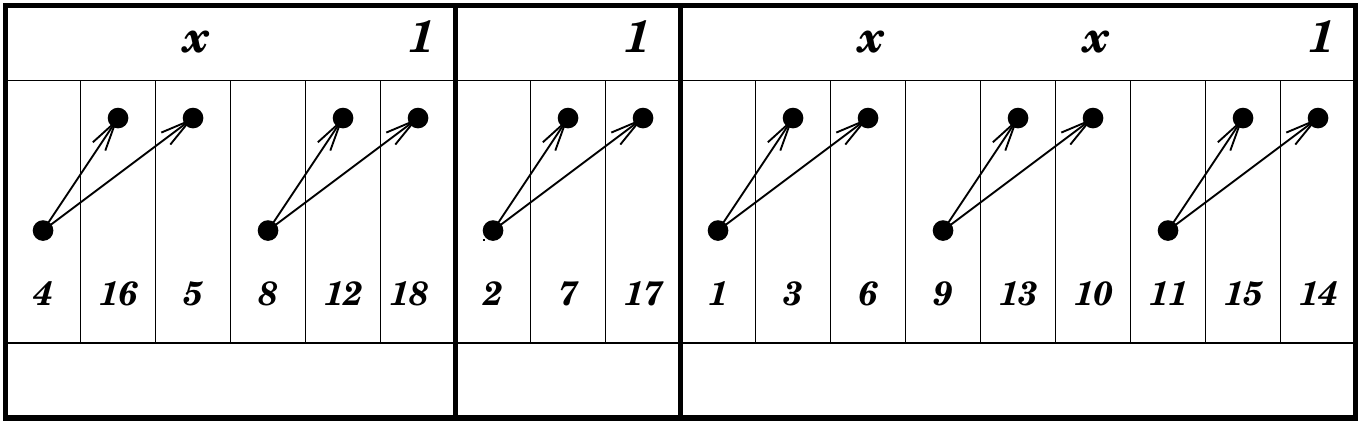}
    \caption{A fixed point of $J_B$.}
    \label{fig:rising3Z}
  \end{center}
\end{figure}

Hence, we have proved that
$$(3n)!\theta_Z(h_{3n}) =
\sum_{F \in \mathcal{F}_n^2} x^{\mathrm{risZ}(F)}$$
as desired.

Thus, for all $Z \in \{T,B,L,A\}$,
\begin{eqnarray*}
\theta_Z(H(t)) &=& 1 + \sum_{n \geq 1}
\frac{t^{3n}}{(3n)!} \sum_{F \in \mathcal{F}_n^2} x^{\mathrm{risZ}(F)} \\
&=& \frac{1}{\theta_Z(E(-t))} = \frac{1}{1+ \sum_{n\geq 1} (-t)^n \theta_Z(e_{n})} \\
&=& \frac{1}{1+ \sum_{n\geq 1} (-t)^{3n} \frac{(-1)^{3n-1}}{(3n)!}
 \mathrm{IZF}_n^2 (x-1)^{n-1}} \\
&=& \frac{1-x}{1-x + \sum_{n \geq 1} \frac{((x-1)t^3)^n}{(3n)!}
\mathrm{IZF}_n^2}.
\end{eqnarray*}

\end{proof}

\section{Computing \texorpdfstring{$\mathrm{IZF}_n^2$}{} for \texorpdfstring{$Z \in \{T,B,L,A\}$}{}}

Based on our results from the last section, all we need to
do in order to compute  the generating functions
$\mathcal{RZ}(x,t)$ for $Z \in \{T,B,L,A\}$ is
to compute  $\mathrm{IZF}_n^2$ for $Z \in \{T,B,L,A\}$.

\subsection{ \texorpdfstring{$\mathrm{ITF}_n^2$}{}}

It is easy to see that if $F =(F_1, \ldots, F_n)$ is such
that $F_1<_T F_2 <_T \cdots <_T F_n$, then the labels on
$F_i$ must be $3i-2$, $3i-1$ and $3i$ for $i =1, \ldots, n$. We have
exactly 2 ways to arrange these labels to make a binary shrub which
are pictured in Figure \ref{fig:22ways}.  It follows
that $\mathrm{ITF}_n^2 =2^n$ for all $n \geq 1$.  Thus
by Theorem \ref{thm:gen},
\begin{eqnarray*}
\mathcal{RT}(x,t) &=& 1 + \sum_{n \geq 1} \frac{t^{3n}}{(3n)!}
\sum_{F \in \mathcal{F}_n^2} x^{\risT(F)} \\
&=&
\frac{1}{1- \sum_{n \geq 1} \frac{t^{3n}}{(3n)!} 2^n (x-1)^{n-1}} \\
&=&
\frac{1-x}{1- x+ \sum_{n \geq 1} \frac{(2(x-1)t^3)^n}{(3n)!}}.
\end{eqnarray*}

\begin{figure}[htbp]
  \begin{center}
    \includegraphics[width=0.3\textwidth,height=1.5cm]{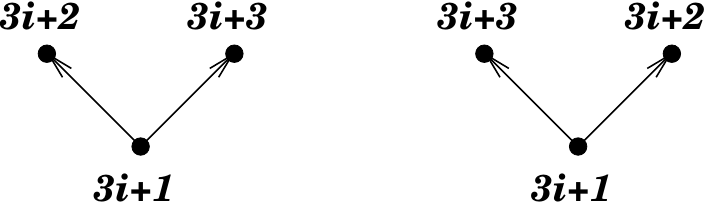}
    \caption{The two ways to label $F_i$ for $F \in \mathcal{ITF}_n^2$.}
    \label{fig:22ways}
  \end{center}
\end{figure}

Using this formula for $\mathcal{RT}(x,t)$, we computed the following
initial terms of
$\mathcal{RT}(x,t)$.
\begin{eqnarray*}
&&1+2\frac{t^3}{3!} +(76+4 x)\frac{t^6}{6!} + (12104+1328 x+8 x^2)
\frac{t^9}{9!} +\\
&& (5048368+843440 x+21776 x^2+16 x^3)\frac{t^{12}}{12!} + \\
&& (4354721312+977383552 x+48921792 x^2+349312 x^3+32 x^4)\frac{t^{15}}{15!}
+ \cdots.
\end{eqnarray*}

\subsection{ \texorpdfstring{$\mathrm{IBF}_n^2$}{}}

The set $\mathcal{IBF}_n^2$ is the set of permutations
that arise from a forest $F = (F_1, \ldots, F_n) \in \mathcal{F}_n^2$
such that the root elements are increasing from left to right.
For example, if $n = 5$, then we are asking for
labellings of the poset whose Hasse diagram is
pictured at the top Figure \ref{fig:posetB} where,
when there is an arrow from a node $x$ to a node $y$, then
the label of node $x$ is  less than label of node $y$.
We have given an example of such a labeling on the second line
of Figure \ref{fig:posetB} and its corresponding permutation
in $\mathcal{SF}_5^2$ in the third line of Figure \ref{fig:posetB}.
Thus we can think of $\mathcal{IBF}_n^2$ as the set of linear extensions
of the poset whose Hasse diagram is of the form pictured in
Figure \ref{fig:posetB}.

\begin{figure}[htbp]
  \begin{center}
    \includegraphics[width=0.5\textwidth,height=4.5cm]{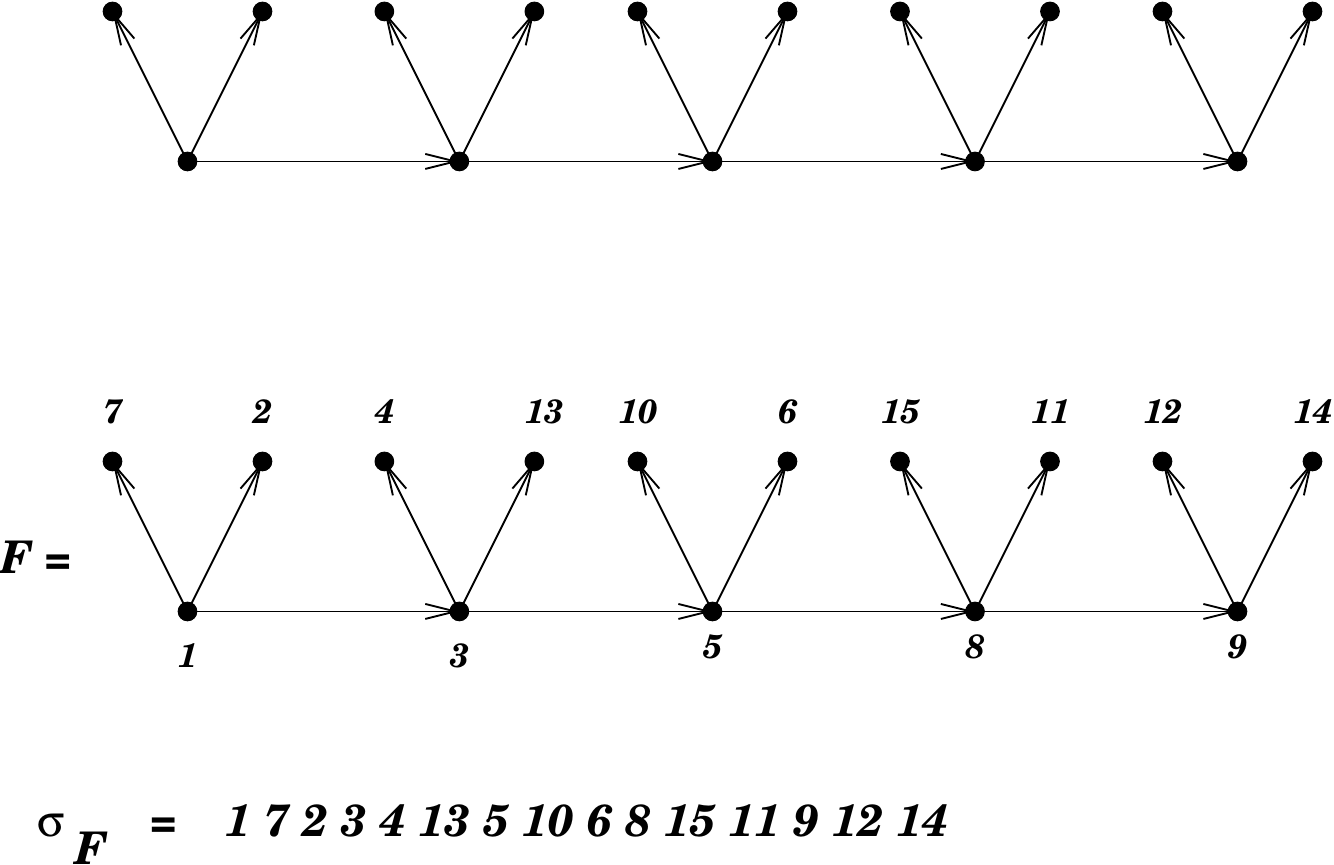}
    \caption{The poset for $\mathcal{IBF}_5^2$.}
    \label{fig:posetB}
  \end{center}
\end{figure}

We claim that
$$\mathrm{IBF}_n^2 = \prod_{k=1}^n 2 \binom{3k-1}{2} = \frac{(3n)!}{3^n (n!)}.$$
This is straightforward to prove by induction.
First, it easy to see from Figure \ref{fig:22ways} that
$$\mathrm{IBF}_1^2 =2 = \frac{3!}{3}.$$
Thus the base case of our induction holds.

Now suppose that our formula holds for $k < n$.
Let $F=(F_1, \ldots, F_n) \in \mathcal{IBF}_n^2$.
Then consider Figure \ref{fig:posetB}.  The label
of the left-most root element must be 1 since there is a directed path
from that element to any other element in the poset. Then we can choose
the remaining two elements in $F_1$ in $\binom{3n-1}{2}$ ways and we have
two ways to order the leaves of $F_1$. Thus we have
$(3n-1)(3n-2)$ ways to pick $F_1$. Once we have picked the labels of
$F_1$, the remaining labels for $F$ are completely free.
It follows
that
\begin{eqnarray*}
\mathrm{IBF}_n^2 &=&(3n-1)(3n-2)\mathrm{IBF}_{n-1}^2 \\
&=& \prod_{k=1}^n (3k-1)(3k-2) = \frac{(3n)!}{3^n (n!)}.
\end{eqnarray*}

Hence,
by Theorem \ref{thm:gen},
\begin{eqnarray*}
\mathcal{RB}(x,t) &=& 1 + \sum_{n \geq 1} \frac{t^{3n}}{(3n)!}
\sum_{F \in \mathcal{F}_n^2} x^{\risB(F)} \\
&=&
\frac{1}{1- \sum_{n \geq 1} \frac{t^{3n}}{(3n)!} \frac{(3n)!}{(3^n(n!))}  (x-1)^{n-1}} \\
&=& \frac{1-x}{1- x+ \sum_{n \geq 1} \frac{(\frac{1}{3}(x-1)t^3)^n}{n!}}\\
&=& \frac{1-x}{-x+ e^{\frac{1}{3}(x-1)t^3}}.
\end{eqnarray*}

Using this formula for $\mathcal{RB}(x,t)$, we computed the following
initial terms of $\mathcal{RB}(x,t)$.

\begin{eqnarray*}
&&1+2\frac{t^3}{3!}+ (40 (1+x)) \frac{t^6}{6!} +
(2240 \left(1+4 x+x^2\right)) \frac{t^9}{9!}+ \\
&& (246400 \left(1+11 x+11 x^2+x^3\right))\frac{t^{12}}{12!} + \\
&& (44844800 \left(1+26 x+66 x^2+26 x^3+x^4\right))\frac{t^{15}}{15!} + \\
&& (12197785600 \left(1+57 x+302 x^2+302 x^3+57 x^4+x^5\right))\frac{t^{18}}{18!} + \cdots .
\end{eqnarray*}

We note that the generating function for rises in permutations
is given by
$$1+\sum_{n \geq 1} \frac{t^n}{n!} \sum_{\sg \in S_n}
x^{\ris(\sg)} = \frac{1-x}{-x+e^{t(x-1)}}.$$
By comparing the form of the generating function $RB(x,t)$, one
can see that
\begin{equation}\label{risred}
\sum_{F \in \mathcal{F}_n^2} x^{\risB(F)} =
\frac{(3n)!}{3^n n!} \sum_{\sg \in S_n} x^{\ris(\sg)}.
\end{equation}
In fact this is easy to prove directly. Suppose that we are given a permutation
$\tau =\tau_1 \cdots \tau_n \in S_n$. Then we claim
that there are $\frac{(3n)!}{3^n n!}$ ways to create
an $F =(F_1, \ldots, F_n) \in \mathcal{F}_n^2$ such that if
$\sg_F = \sg_1 \cdots \sg_{3n}$, then
$\red(\sg_1 \sg_4  \cdots \sg_{3n-2}) = \tau$. That is,
suppose that $\tau_{j_k} =k$ for $k =1, \ldots, n$.
We let $1$ be the label of
the root of $F_{j_1}$ and then we have $(3n-1)(3n-2)$ ways
to pick the right and left leaves of $F_{j_1}$.  Once
we have fixed $F_{j_1}$, we let $c_2$ be the smallest
element $c$ in $\{1, \ldots, 3n\}$ such that $c$ is not a label
in $F_{j_1}$. We label the root of $F_{j_2}$ with $c_2$ and
then we have $(3n-4)(3n-5)$ ways to
pick the right and left leaves of $F_{j_2}$. Once
we have fixed $F_{j_1}$ and $F_{j_2}$, we let $c_3$ be the smallest
element $c$ in $\{1, \ldots, 3n\}$ such that $c$ is not a label
in $F_{j_1}$ or $F_{j_2}$. We label the root of $F_{j_3}$ with $c_3$ and
then we have $(3n-7)(3n-8)$ ways to
pick the right and left leaves of $F_{j_3}$. Continuing on in this
way, we see that there are $\prod_{i=0}^{n-1} (3n-(3k+1))(3n-(3k+2)) =
\frac{(3n)!}{3^n n!}$ ways to create an $F =(F_1, \ldots, F_n) \in \mathcal{F}_n^2$ such that if $\sg_F = \sg_1 \cdots \sg_{3n}$, then
$\red(\sg_1 \sg_4  \cdots \sg_{3n-2}) = \tau$. Observe that for any
$F$ created in this way, $\risB(F) = \ris(\tau)$. Thus (\ref{risred}) easily follows.

\subsection{ \texorpdfstring{$\mathrm{ILF}_n^2$}{}}

The set $\mathcal{ILF}_n^2$ is the set
of forests $F = (F_1, \ldots, F_n) \in \mathcal{F}_n^2$
such that $F_1<_L F_2 <_L <_L \cdots <_L F_n$.
Such a forest can be considered to a be labeling of a poset
$\mathbb{L}_{3n}$ of the type whose Hasse diagram is
pictured in Figure \ref{fig:risL}. For example,
at the bottom of Figure \ref{fig:risL}, we have redrawn the poset in a nicer
form.  Here when we draw an arrow from node $x$ to node $y$, then
we want the label of node $x$ to be less than label of node $y$ in
$\mathbb{L}_{3n}$. Thus the Hasse diagram of
$\mathbb{L}_{3n}$ consists of 3 rows of
$n$ nodes such that there are arrows
connecting the nodes in each row which go from left to right and, in each
column, there are arrows going from the node in the middle row to the nodes
at the top and bottom of that column.
Let $\mathcal{L}_{3n}$ denote the set of all linear extensions
of $\mathbb{L}_{3n}$, that is, the set of all
labellings of $\mathbb{L}_{3n}$ with the numbers
$\{1, \ldots, 3n\}$ such that if there is an arrow from node $x$ to
$y$, then the label on node $x$ is less than the label on node $y$.
Thus $\mathrm{ILF}_n^2 =|\mathcal{L}_{3n}|$.

\begin{figure}[htbp]
  \begin{center}
    \includegraphics[width=0.4\textwidth,height=3cm]{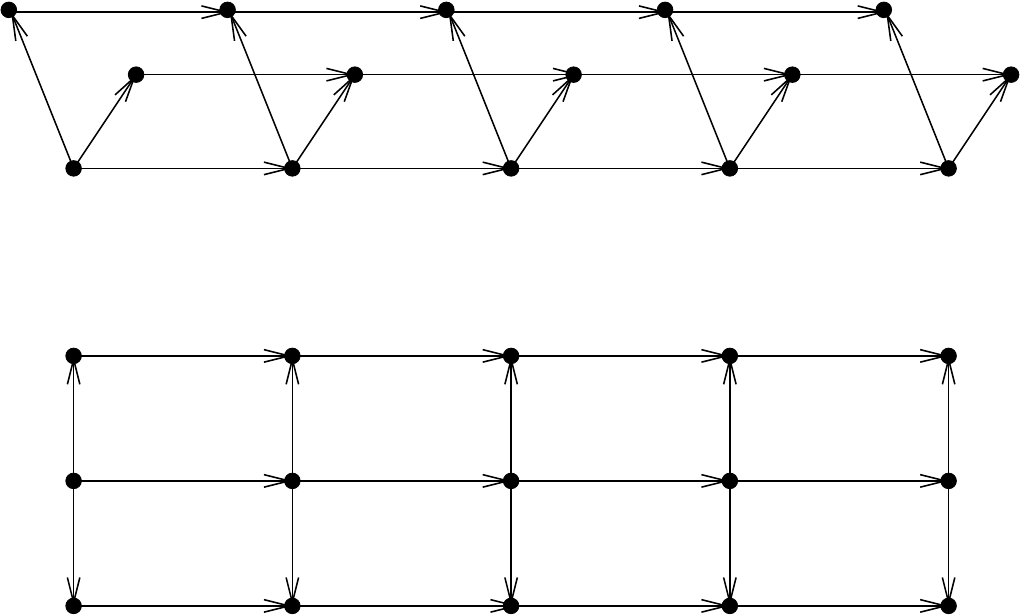}
    \caption{The poset for $\mathcal{ILF}_5^2$.}
    \label{fig:risL}
  \end{center}
\end{figure}

We then have the following theorem.

\begin{theorem}
$\mathrm{ILF}_n^2 = \frac{4^n (3n)!}{(n+1)!(2n+1)!}$.
\end{theorem}
\begin{proof}
In \cite{K65} it has been proved that
$\frac{4^n (3n)!}{(n+1)!(2n+1)!}$ is the number of paths $P=(p_1,
\ldots, p_{3n})$ in the plane
which start at (0,0) and end at (0,0), stay entirely
in the first quadrant, and use only northeast steps $(1,1)$,
west steps $(-1,0)$, and south steps $(0,-1)$. See also
\cite{MBM} and \cite{Gessel}. The fact
that $P$ starts and ends at $(0,0)$ means that $P$ has
$n$ northeast steps, $n$ west steps, and
$n$ south steps. For any $1 \leq i \leq 3n$,
let $NE_i(P)$ equal the number of northeast steps in $(p_1, \ldots, p_i)$,
$W_i(P)$ equal the number of west steps in $(p_1, \ldots, p_i)$, and
$S_i(P)$ equal the number of south steps in $(p_1, \ldots, p_i)$.
The fact that $P$ stays in the first quadrant is equivalent
to the conditions that $NE_i(P) \geq W_i(P)$ and
$NE_i(P) \geq S_i(P)$ for $i =1, \ldots, 3n$. Let $\mathcal{P}_{3n}$ denote
the set of all such paths $P$ of length $3n$.

To prove our theorem, we shall define a bijection
from $\Gamma:\mathcal{L}_{3n} \rightarrow \mathcal{P}_{3n}$.
The map $\Gamma$ is quite simple, given a labeling $L \in
\mathcal{L}_{3n}$, we let $\Gamma(L) = (p_1, \ldots, p_{3n})$ be
the path which starts at (0,0) and
where $p_i$ is a northeast step if the label $i$ is in
the middle row of $L$, $p_i$ is west step if the label $i$ is in the top
row of $L$, and $p_i$ is a south step if the label $i$ is
in the bottom row $L$. An example of this map is given
in Figure \ref{fig:risL2} where we have put a label $i$ on
the $i^{th}$-step of the $\Gamma(L)$.

\begin{figure}[htbp]
  \begin{center}
    \includegraphics[width=0.3\textwidth,height=4.5cm]{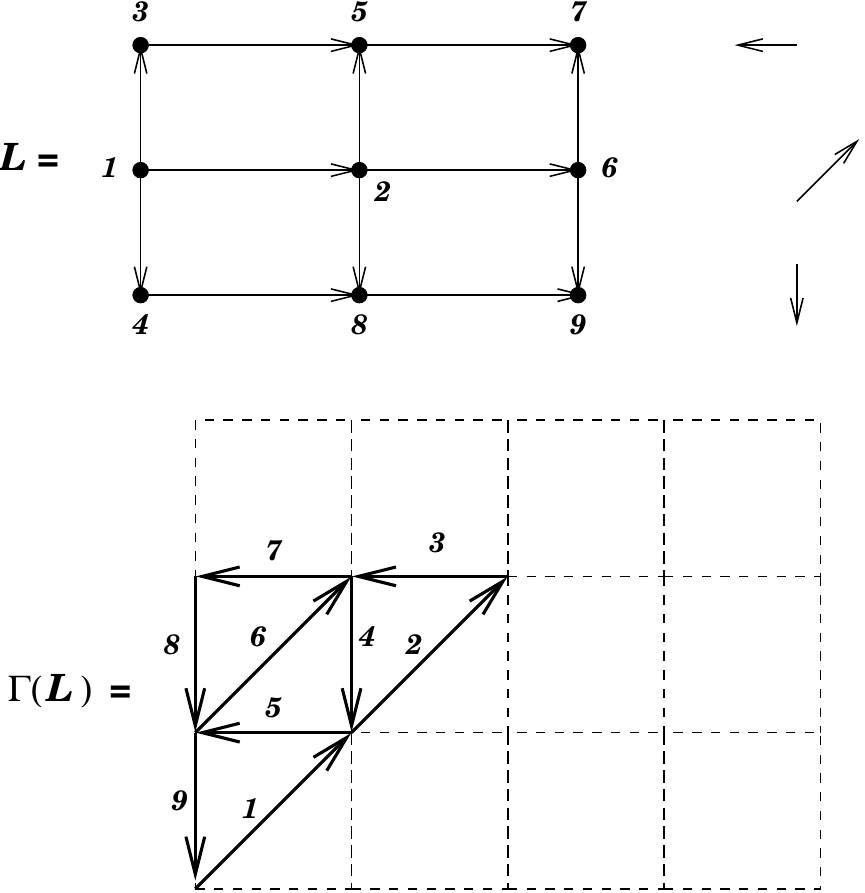}
    \caption{The bijection $\Gamma$.}
    \label{fig:risL2}
  \end{center}
\end{figure}

First we must check that if $L \in \mathcal{L}_{3n}$, then
$\Gamma(L) = (p_1, \ldots, p_{3n})$ is an element
of $\mathcal{P}_{3n}$.  It is easy to see that $\Gamma(L)$
starts and ends at $(0,0)$ since $\Gamma(L)$ has
$n$ northeast steps, $n$ west steps, and $n$ south steps.
Let $LT_i$, $LM_i$, and $LB_i$ denote
the label in $L$ of the $i^{th}$ element of the top row, middle row,
and bottom row, reading from left to right, respectively.
Suppose for a contradiction that there is
an $t$ such that $i = W_t(P) > NE_t(P) =j$.  This is impossible
since this would imply that $LT_i \leq t$ and $LM_i > t$ which
violates that fact that there is an arrow from the element
in the middle row of the $i^{th}$-column to the element in
the top row of $i^{th}$-column in $\mathbb{L}_{3n}$.
Similarly, suppose that there is
an $t$ such that $i = S_t(P) > NE_t(P) =j$.  This is impossible
since this would imply that $LB_i \leq t$ and $LM_i > t$ which
violates that fact that there is an arrow from the element
in the middle row of the $i^{th}$-column to the element in
the bottom row of $i^{th}$-column in $\mathbb{L}_{3n}$.   Thus for all $t$,
$NE_t(\Gamma(L)) \geq W_t(\Gamma(L))$ and
$NE_t(\Gamma(L)) \geq S_t(\Gamma(L))$ which means that
$\Gamma(L)$ stays in the first quadrant.

It is easy to see that $\Gamma$ is one-to-one.  That is, if
$L$ and $L'$ are two different labellings in $\mathcal{L}_{3n}$,
then let $i$ be the least $j$ such that $j$ is not in the same position
in the labellings $L$ and $L'$. Then clearly,
$\Gamma(L) \neq \Gamma(L')$ since the $i^{th}$ step of $\Gamma(L)$ will not be
the same as the $i^{th}$ step of $\Gamma(L')$.  To see that
$\Gamma$ maps onto $\mathcal{P}_{3n}$, suppose that we are given
$P=(p_1, \ldots, p_{3n})$ in $\mathcal{P}_{3n}$.
Let $L$ be the labeling of $\mathbb{L}_{3n}$ which is
increasing in the rows of $\mathbb{L}_{3n}$ such that
$i$ is label in the top row of $\mathbb{L}_{3n}$ if $p_i$ is a
west step, $i$ is label in the middle row of $\mathbb{L}_{3n}$ if $p_i$ is a
northeast step, and $i$ is label in the bottom
row of $\mathbb{L}_{3n}$ if $p_i$ is a south step. It is easy to
see from our definitions that $\Gamma(L) =P$. Hence the only thing
that we have to do is to check that $L \in \mathcal{L}_{3n}$.
Since $L$ is increasing in rows, we need only check that
the for each column $i$, the label $x$ of the
element in the middle row of column
$i$ is less than the label $y$ of the element in the top row of
column $i$ and and less than the label $z$ of the element of the bottom
row of column $i$. But this follows from the fact that
$P$ stays in the first quadrant. That is, if $y < x$, then
in $(p_1, \ldots, p_y)$, we would have more west steps than northeast steps
which would mean that the $y^{th}$ step of $P$ is not in the first quadrant.
Similarly if $z < x$, then
in $(p_1, \ldots, p_z)$, we would have more south steps than northeast steps
which would mean that the $z^{th}$ step of $P$ is not in the first quadrant.
Thus $\Gamma$ is a bijection from $\mathcal{L}_{3n}$ onto
$\mathcal{P}_{3n}$.

\end{proof}

Hence,
by Theorem \ref{thm:gen},
\begin{eqnarray*}
\mathcal{RL}(x,t) &=& 1 + \sum_{n \geq 1} \frac{t^{3n}}{(3n)!}
\sum_{F \in \mathcal{F}_n^2} x^{\risL(F)} \\
&=&
\frac{1}{1- \sum_{n \geq 1} \frac{t^{3n}}{(3n)!} \frac{4^n ((3n)!)}{(n+1)! (2n+1)!}  (x-1)^{n-1}} \\
&=& \frac{1-x}{1- x+ \sum_{n \geq 1} \frac{(4(x-1)t^3)^n}{(n+1)!(2n+1)!}}.
\end{eqnarray*}

Using this formula for $\mathcal{RL}(x,t)$, we computed
the following initial terms of $\mathcal{RL}(x,t)$.
\begin{eqnarray*}
&&1 +2\frac{t^3}{3!}+ 16 (4+x)\frac{t^6}{6!}+ 192 \left(43+26 x+x^2\right)
\frac{t^9}{9!}+  \\
&& 2816 \left(983+975 x+141 x^2+x^3\right)\frac{t^{12}}{12!}+ \\
&& 46592 \left(41141+57086 x+16506 x^2+766 x^3+x^4\right)\frac{t^{15}}{15!}+ \\
&& 835584 \left(2848169+5084786 x+2311247 x^2+261973 x^3+4324 x^4+x^5\right)\frac{t^{18}}{18!}+ \cdots .
\end{eqnarray*}

\subsection{\texorpdfstring{$\mathrm{IAF}_n^2$}{}.}

As with our other examples, we can think of $\mathrm{IAF}_n^2$ as the number of linear
extensions
of a poset of the type whose Hasse diagram is
pictured at the top of Figure \ref{fig:posetsA}.
That is, the Hasse diagram of
$A_n$ consists of $n$ binary shrubs where
there is an arrow from the right-most element of each shrub to the
left-most element of the next shrub.  We
shall also need to consider three related posets, $E_n$, $S_n$, and $B_n$.
$E_n$ is the poset whose Hasse diagram starts
with the Hasse diagram of
$A_n$ and adds one extra node which is connected to the Hasse diagram of $A_n$
by an arrow that goes from the right-most node of the right-most binary shrub
to the new node. $S_n$ is the poset whose Hasse diagram starts
with Hasse diagram of $A_n$ and adds one extra node which is connected to the Hasse diagram
of $A_n$
by an arrow that goes from the new node to the left-most node of the left-most binary shrub. $B_n$ is the poset whose Hasse diagram
starts with the Hasse diagram of $A_n$ and adds two
extra nodes, one which is connected as in $E_n$ and one which is connected
as in $S_n$.  Thus the Hasse diagram of
$E_n$ starts with Hasse diagram of $A_n$ and adds an extra node at the end,
the Hasse diagram of $S_n$ starts with the Hasse diagram of $A_n$ and adds an extra node at the start, and the Hasse diagram of $B_n$ starts with
the Hasse diagram of $A_n$ and adds both an extra node at the end and an extra node at the start.
For example, Figure \ref{fig:posetsA} pictures $A_5$, $E_5$, $S_5$,
and $B_5$.

\begin{figure}[htbp]
  \begin{center}
    \includegraphics[width=0.5\textwidth,height=3.2cm]{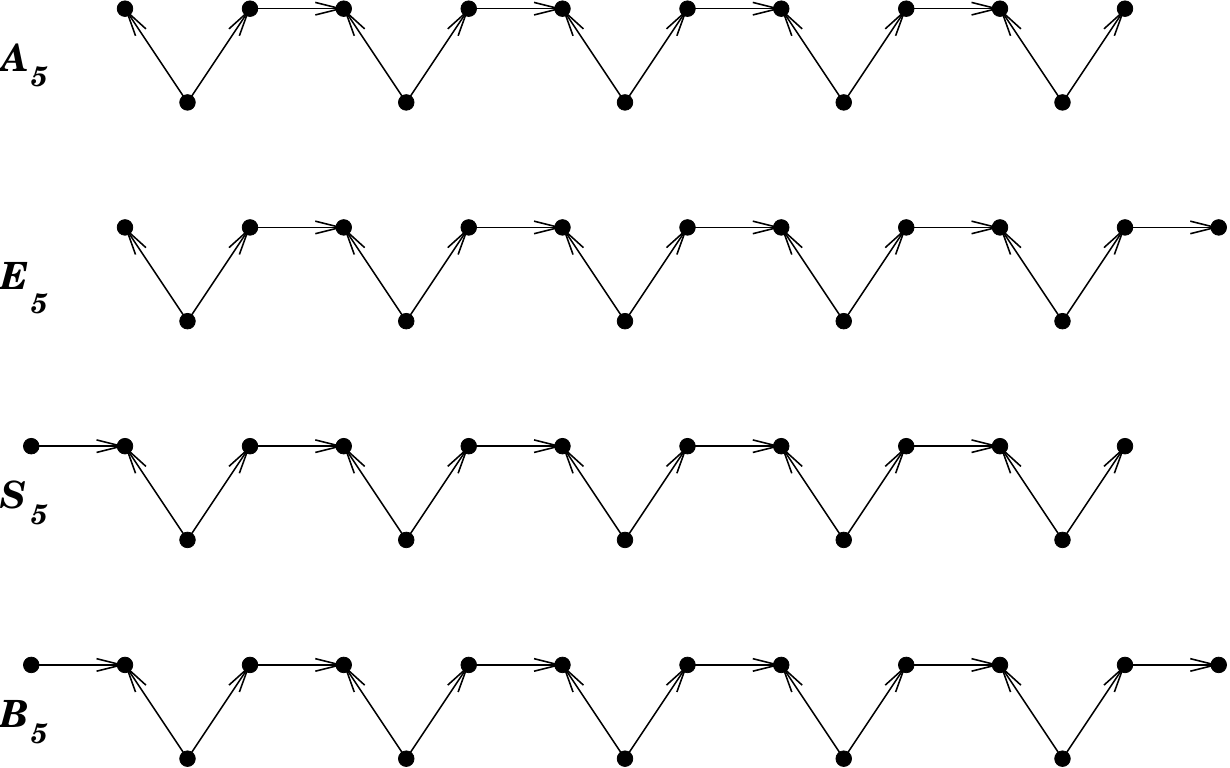}
    \caption{The posets $A_5$, $E_5$, $S_5$, and $B_5$.}
    \label{fig:posetsA}
  \end{center}
\end{figure}

For $W \in \{A,E,S,B\}$, we let $\mathcal{LW}_n$ denote the set
of linear extensions of $W_n$ and $\mathrm{LW}_n =|\mathcal{LW}_n|$.
We shall show that
$\mathrm{LA}_n$, $\mathrm{LE}_n$, $\mathrm{LS}_n$, and
$\mathrm{LB}_n$ satisfy simple recurrence relations.
First in Figures \ref{fig:A1} and \ref{fig:B1}, we
have listed all the elements of $\mathcal{LA}_1$, $\mathcal{LE}_1$,
$\mathcal{LS}_1$, and $\mathcal{LB}_1$. Thus
\begin{equation*}
\mathrm{LA}_1 = 2, \ \mathrm{LE}_1 = 3, \
\mathrm{LS}_1 = 5, \ \mbox{and} \
\mathrm{LB}_1 = 9.
\end{equation*}

\begin{figure}[htbp]
  \begin{center}
    \includegraphics[width=0.4\textwidth,height=3.5cm]{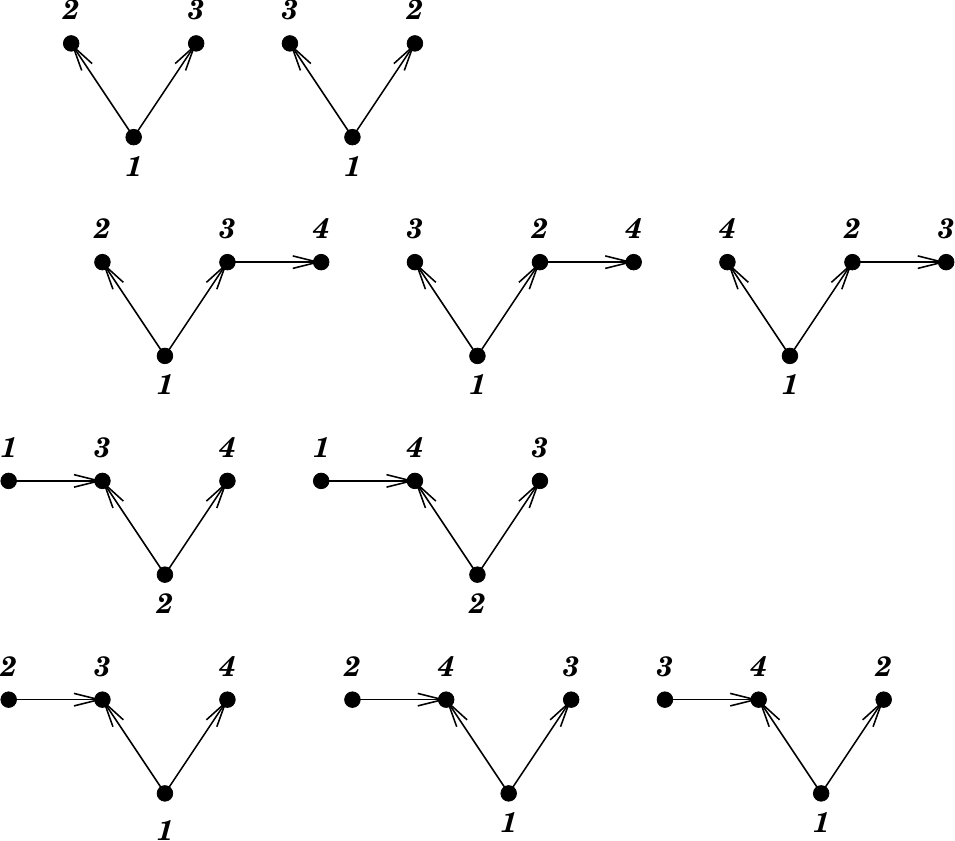}
    \caption{The elements of $\mathcal{LA}_1$, $\mathcal{LE}_1$,
and $\mathcal{LS}_1$.}
    \label{fig:A1}
  \end{center}
\end{figure}
\begin{figure}[htbp]
  \begin{center}
  \includegraphics[width=0.5\textwidth,height=3.2cm]{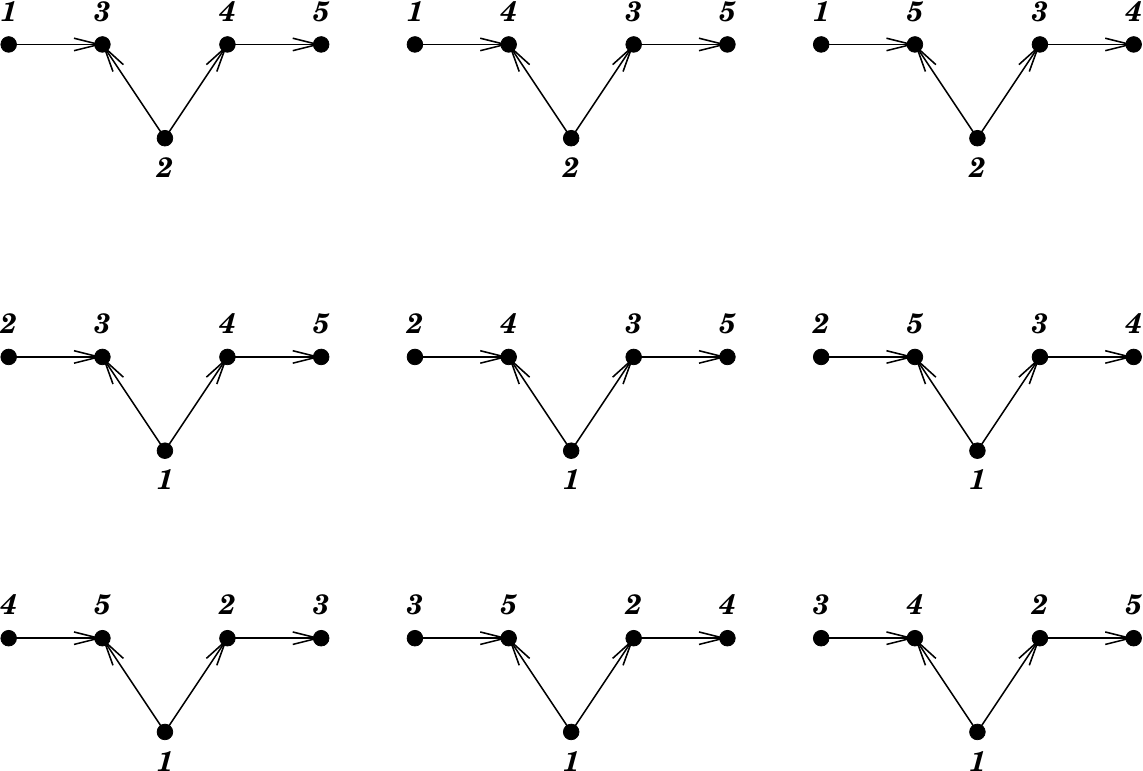}
  \caption{The elements of $\mathcal{LB}_1$.}
  \label{fig:B1}
  \end{center}
\end{figure}

We start with the recursion for $\mathrm{LB}_n$.
Suppose that $n > 1$. Then consider where the label 1 can be in
an element of $\mathcal{LB}_n$. There are four cases to consider.
First, 1 could be the label of the left-most
element in which case the remaining labels must correspond to a linear extension
of $E_n$. Otherwise 1 is the label of the root of the $k^{th}$ binary
shrub for some $k =1, \ldots, n$. If $1< k < n$, then
there is no relation that is forced between
the labels to left of 1 which correspond to a linear extension  of
$B_{k-1}$ and the labels to the right of 1 which correspond to a linear extension
of $B_{n-k}$.  In the special case where $k=1$, the Hasse diagram
of the poset to the left of the node labeled 1 is just a 2 element
chain which we call $B_0$. Similarly, in
special case where $k=n$, the Hasse diagram
of the poset to the right of the node labeled 1 is just $B_0$. Clearly,
$\mathrm{LB}_0 =1$.  These four cases are pictured in
Figure \ref{fig:Bn}. For each $k =1, \ldots, n$,
we have $\binom{3n+1}{3(k-1)+2}$ ways to choose the labels of the
elements to the left of 1. It follows that
\begin{equation}\label{Bnrec}
\mathrm{LB}_n = \mathrm{LE}_n + \sum_{k=1}^n \binom{3n+1}{3(k-1)+2}
\mathrm{LB}_{k-1}\mathrm{LB}_{n-k}.
\end{equation}

\begin{figure}[htbp]
  \begin{center}
  \includegraphics[width=0.5\textwidth,height=3.5cm]{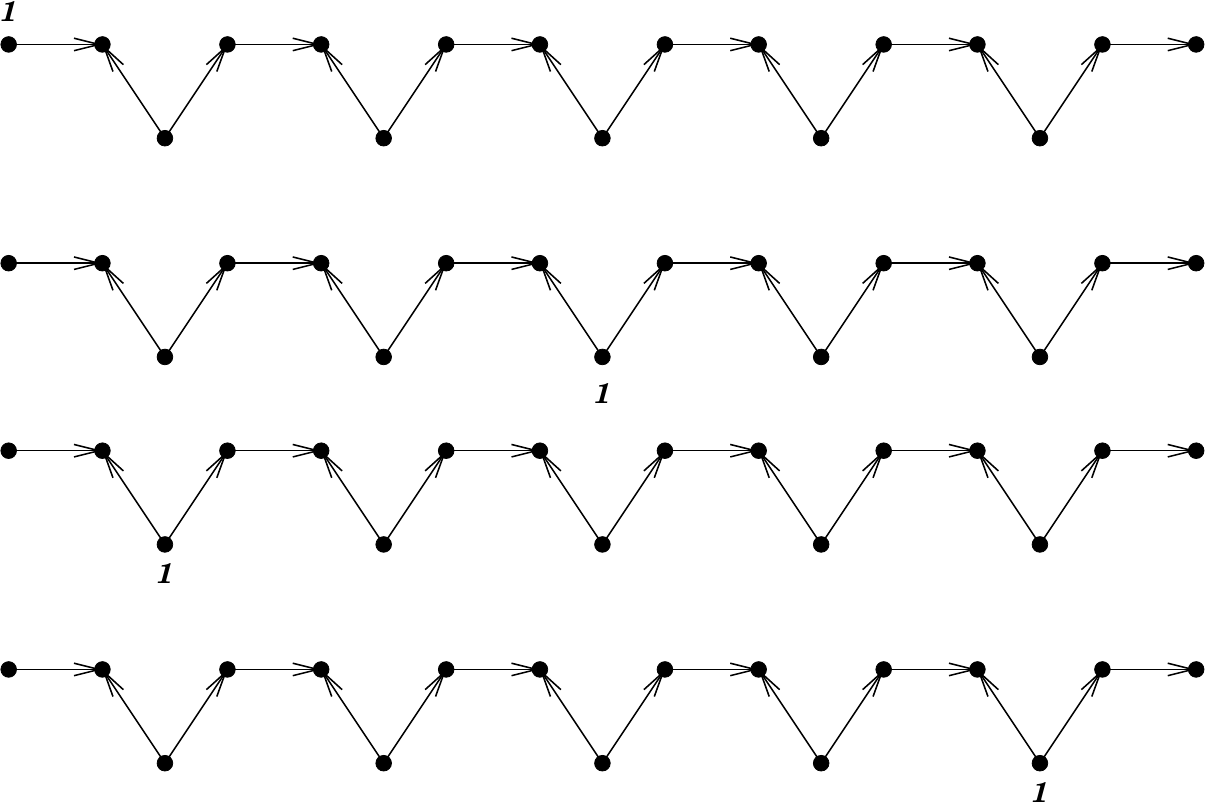}
  \caption{The recursion for $\mathrm{LB}_n$.}
  \label{fig:Bn}
  \end{center}
\end{figure}

Next consider the recursion for $\mathrm{LS}_n$.
Suppose that $n > 1$. Then consider where the label 1 can be in
an element of $\mathcal{LS}_n$. Again there are four cases to consider.
First, 1 could be the label of the left-most
element in which case the remaining labels must correspond to a linear extension
of $A_n$. Otherwise 1 is the label of the root of the $k^{th}$ binary
shrub for some $k =1, \ldots, n$. If $1< k < n$, then
there is no relation that is forced between
the labels to left of 1 which correspond to a linear extension  of
$B_{k-1}$ and the labels to the right of 1 which correspond to a linear extension
of $S_{n-k}$.  In the special case where $k=1$, the Hasse diagram
of the poset to the left of the node labeled 1 is just $B_0$. Similarly, in
special case where $k=n$, the Hasse diagram
of the poset to the right of the node labeled 1 is a one element
poset which we call $S_0$. Clearly,
$\mathrm{LS}_0 =1$.  These four cases are pictured in
Figure \ref{fig:Sn}. For each $k =1, \ldots, n$,
we have $\binom{3n}{3(k-1)+2}$ ways to choose the labels of the
elements to the left of 1. It follows that got $n \geq 2$,
\begin{equation}\label{Snrec}
\mathrm{LS}_n = \mathrm{LA}_n + \sum_{k=1}^n \binom{3n}{3(k-1)+2}
\mathrm{LB}_{k-1}\mathrm{LS}_{n-k}.
\end{equation}

\begin{figure}[htbp]
  \begin{center}
  \includegraphics[width=0.5\textwidth,height=3.5cm]{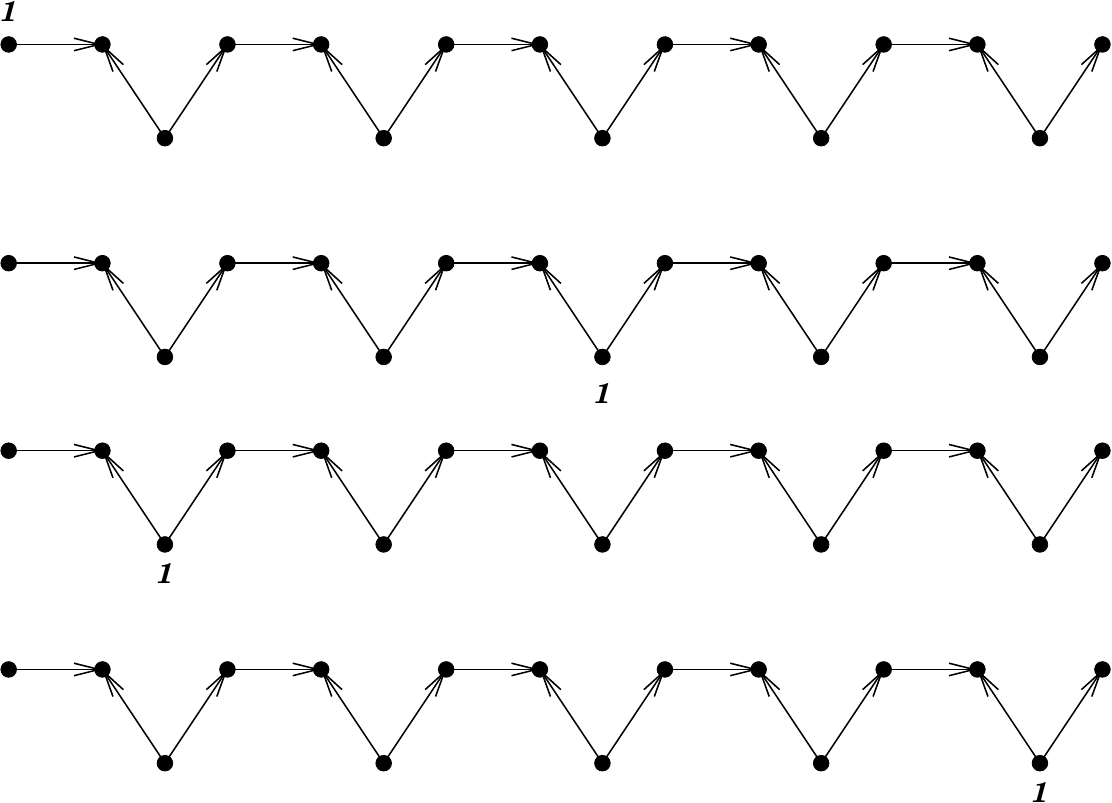}
  \caption{The recursion for $\mathrm{LS}_n$.}
  \label{fig:Sn}
  \end{center}
\end{figure}

Next consider the recursion for $\mathrm{LE}_n$.
Suppose that $n > 1$. Then consider where the label 1 can be in
an element of $\mathcal{LS}_n$. In this case,
there are three  cases to consider. That is,
 1 must be the label of the root of the $k^{th}$ binary
shrub for some $k =1, \ldots, n$. If $1< k < n$, then
there is no relation that is forced between
the labels to left of 1 which correspond to a linear extension  of
$E_{k-1}$ and the labels to the right of 1 which correspond to a linear extension
of $B_{n-k}$.  In the special case where $k=1$, the Hasse diagram
of the poset to the left of the node labeled 1 is just a one element
poset which we will also call $E_0$. Clearly,
$\mathrm{LE}_0 =1$. Similarly, in
special case where $k=n$, the Hasse diagram
of the poset to the right of the node labeled 1 is just $B_0$.
These three cases are pictured in
Figure \ref{fig:En}. For each $k =1, \ldots, n$,
we have $\binom{3n}{3(k-1)+1}$ ways to choose the labels of the
elements to the left of 1. It follows that
\begin{equation}\label{Enrec}
\mathrm{LE}_n = \sum_{k=1}^n \binom{3n}{3(k-1)+1}
\mathrm{LE}_{k-1}\mathrm{LB}_{n-k}.
\end{equation}

\begin{figure}[htbp]
  \begin{center}
  \includegraphics[width=0.5\textwidth,height=3cm]{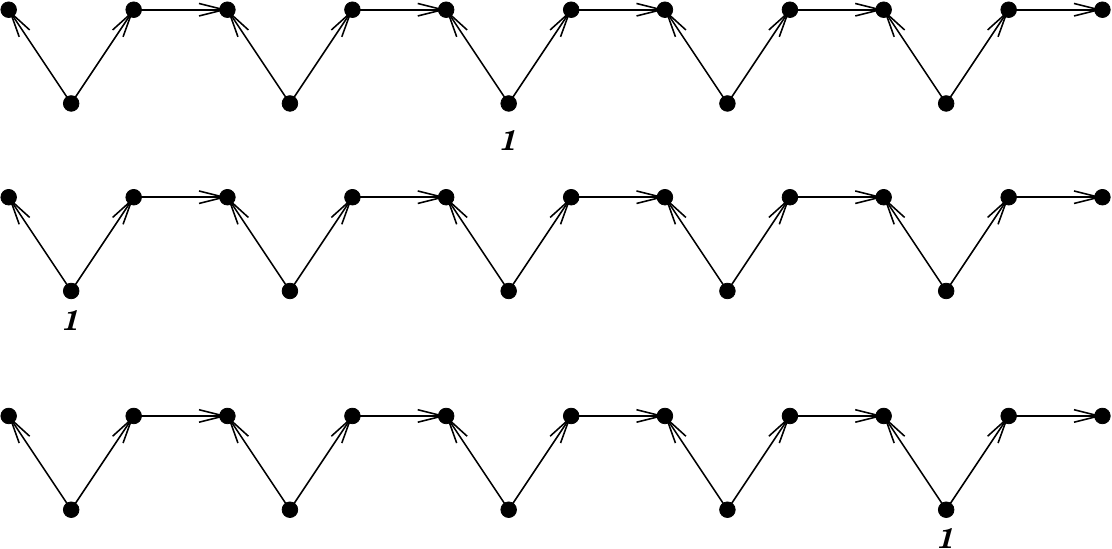}
  \caption{The recursion for $\mathrm{LE}_n$.}
  \label{fig:En}
  \end{center}
\end{figure}

Finally consider the recursion for $\mathrm{LA}_n$.
Now suppose that $n > 1$. Then consider where the label 1 can be in
an element of $\mathcal{LS}_n$. In this case,
there are three  cases to consider. That is,
 1 must be the label of the root of the $k^{th}$ binary
shrub for some $k =1, \ldots, n$. If $1< k < n$, then
there is no relation that is forced between
the labels to left of 1 which correspond to a linear extension  of
$E_{k-1}$ and the labels to the right of 1 which correspond to a linear extension
of $S_{n-k}$.  In the special case where $k=1$, the Hasse diagram
of the poset to the left of the node labeled 1 is $E_0$. Similarly, in
special case where $k=n$, the Hasse diagram
of the poset to the right of the node labeled 1 is just $S_0$.
These three cases are pictured in
Figure \ref{fig:An}. For each $k =1, \ldots, n$,
we have $\binom{3n-1}{3(k-1)+1}$ ways to choose the labels of the
elements to the left of 1. It follows that for $n \geq 2$,
\begin{equation}\label{Anrec}
\mathrm{LA}_n = \sum_{k=1}^n \binom{3n-1}{3(k-1)+1}
\mathrm{LE}_{k-1}\mathrm{LS}_{n-k}.
\end{equation}

\begin{figure}[htbp]
  \begin{center}
  \includegraphics[width=0.5\textwidth,height=3cm]{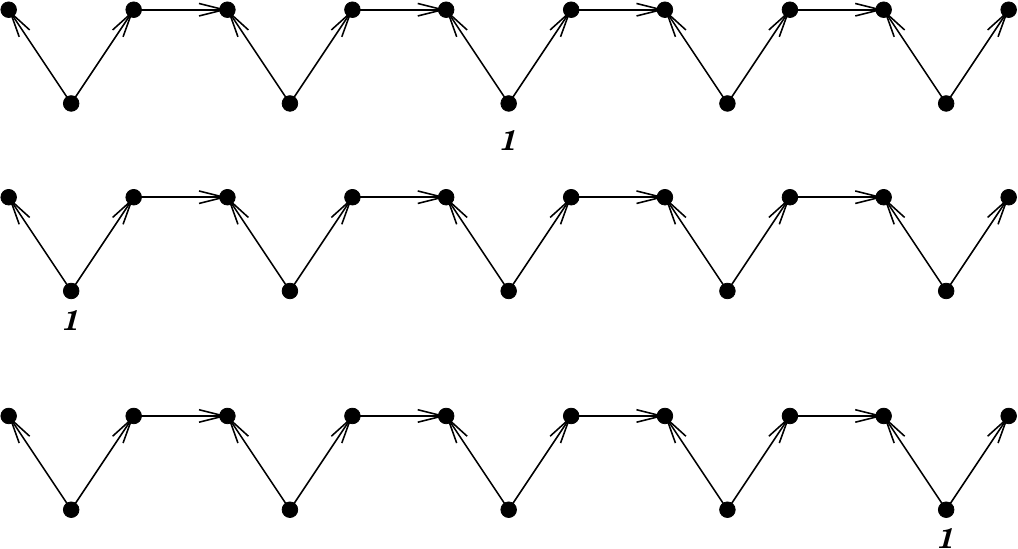}
  \caption{The recursion for $\mathrm{LA}_n$.}
  \label{fig:An}
  \end{center}
\end{figure}

One can check directly that (\ref{Bnrec}), (\ref{Snrec}),
(\ref{Enrec}), and (\ref{Anrec}) also hold for $n=1$.
By iterating these recursions, we can compute
the first few terms of the sequences
$(\mathrm{LA}_n)_{n \geq 0}$, $(\mathrm{LB}_n)_{n \geq 0}$,
$(\mathrm{LE}_n)_{n \geq 0}$, and $(\mathrm{LS}_n)_{n \geq 0}$.
For example, the first few terms of $(\mathrm{LA}_n)_{n \geq 0}$ are
\begin{eqnarray*}
&&1,2,40,3194,666160,287316122,222237912664,280180369563194,\\
&&537546603651987424,1490424231594917313242,5735930050702709579598280, \ldots
\end{eqnarray*}
The first few terms of $(\mathrm{LB}_n)_{n \geq 0}$ are
\begin{eqnarray*}
&&1,9,477,74601,25740261,16591655817,17929265150637,30098784753112329,\\
&&74180579084559895221,256937013876000351610089,1208025937371403268201735037, \ldots
\end{eqnarray*}
The first few terms of $(\mathrm{LE}_n)_{n \geq 0}$ are
\begin{eqnarray*}
 &&1,3,99,11259,3052323,1620265923,1488257158851,2172534146099019,\\
&&4736552519729393091,14708695606607601165843,62671742039942099631403299,
\ldots
\end{eqnarray*}
The first few terms of $(\mathrm{LS}_n)_{n \geq 0}$ are
\begin{eqnarray*}
&&1,5,169,19241,5216485,2769073949,2543467934449,3712914075133121,\\
&&8094884285992309261,25137521105896509819605,107107542395866078895709049
\ldots
\end{eqnarray*}
None of these sequences appear in the OEIS, see \cite{oeis}.

One can also study the generating functions
\begin{eqnarray*}
\mathcal{A}(t) &=& 1 + \sum_{n \geq 1} \frac{\mathrm{LA}_n t^{3n}}{(3n)!}, \\
\mathcal{E}(t) &=& \sum_{n \geq 0} \frac{\mathrm{LE}_n t^{3n+1}}{(3n+1)!}, \\
\mathcal{S}(t) &=& \sum_{n \geq 0} \frac{\mathrm{LS}_n t^{3n+1}}{(3n+1)!},
\ \mbox{and} \\
\mathcal{B}(t) &=& \sum_{n \geq 0} \frac{\mathrm{LB}_n t^{3n+2}}{(3n+2)!}.
\end{eqnarray*}
It is straightforward to show that the recursions (\ref{Bnrec}), (\ref{Enrec}),
(\ref{Snrec}), and (\ref{Anrec}) imply that the following differential
equations hold:
\begin{eqnarray*}
\mathcal{A}^\prime(t) &=& \mathcal{E}(t) \mathcal{S}(t), \\
\mathcal{E}^\prime(t) &=& 1 +\mathcal{E}(t) \mathcal{B}(t), \\
\mathcal{S}^\prime(t) &=& \mathcal{A}(t) + \mathcal{B}(t) \mathcal{S}(t),
\ \mbox{and} \\
\mathcal{B}^\prime(t) &=& t+ \mathcal{E}(t) + (\mathcal{B}(t))^2.
\end{eqnarray*}

Note that it follows from the last differential equation that
$$\mathcal{B}^\prime(t) - t - (\mathcal{B}(t))^2 = \mathcal{E}(t),$$
which can be plugged into the second differential equation to show that
\begin{equation}\label{Bdiffeq}
\mathcal{B}^{\prime\prime}(t) = 2+ 3 \mathcal{B}^\prime(t)
\mathcal{B}(t) -t \mathcal{B}(t) - (\mathcal{B}(t))^3.
\end{equation}
Thus in principle, we can obtain a recursion for
the $\mathrm{LB}_n$ in terms of $\mathrm{LB}_0, \ldots, \mathrm{LB}_{n-1}$
which in turn can lead to more direct recursions for
$\mathrm{LE}_n$,  $\mathrm{LS}_n$, and $\mathrm{LA}_n$. However,
all such  recursions are more complicated than the family of recursions
described above.

We used the initial terms of the sequence $(\mathrm{LA}_n)_{n \geq 0}$ to
compute the following initial terms of $\mathcal{RA}(x,t)$.

\begin{eqnarray*}
&&1+ 2 \frac{t^3}{3!} + 40 (1+x)\frac{t^6}{6!} +(3194+7052 x+3194 x^2)
\frac{t^9}{9!}+ \\
&& 880 \left(757+2603 x+2603 x^2+757 x^3\right) \frac{t^{12}}{12!}+ \\
&& 2 \left(143658061+671012156 x+1061347566 x^2+671012156 x^3+143658061 x^4\right) \frac{t^{15}}{15!}+  \\
&& 136 \left(1634102299+9646627503 x+21007526198 x^2+21007526198 x^3+\right.\\
&& \ \ \  \left. 9646627503 x^4+1634102299 x^5\right)\frac{t^{18}}{18!}
+ \cdots .
\end{eqnarray*}

\section{Conclusions}

In this paper, we computed the generating function of 5 different kinds of
rises in forests of binary shrubs. Our work can be viewed as the first step in studying
consecutive patterns in forests of binary shrubs. We will study such patterns
in a subsequent paper.

In addition, we can also study the analogues of
up-down permutations relative to $<_T$, $<_B$, $<_L$ and
$<_A$. For example, we say that an $F =(F_1, \ldots, F_n) \in
\mathcal{F}^2_n$ is an up-down forest of binary shrubs
with respect to the $<_{T}$ if
$RiseT(F)$ equals the set of odd numbers less than $n$. We also will
study such analogues of up-down permutations in a subsequent paper.

\acknowledgements
\label{sec:ack}
We sincerely thank the anonymous reviewers for valuable comments and their careful reading of our
paper, which were of great help in revising the manuscript.

\bibliographystyle{abbrvnat}
\bibliography{Final}
\label{sec:biblio}

\end{document}